\documentclass{amsart}
\usepackage{amsfonts}
\usepackage{amsmath}
\usepackage{amsthm}
\usepackage{amssymb}
\usepackage{color}
\usepackage[margin=1.0in]{geometry}
\usepackage{hyperref}
\usepackage{mathrsfs}% for script fonts
\usepackage{caption}
\usepackage{subcaption}
\usepackage{tikz}\usetikzlibrary{decorations.pathreplacing}

\def\@begintheorem#1#2{\par\bgroup{\sc #1\ #2. }\it\ignorespaces}
\def\@opargbegintheorem#1#2#3{\par\bgroup{\sc #1\ #2\ (#3). } \it\ignorespaces}
\def\@endtheorem{\egroup}
\newtheoremstyle{theoremstyleslant}{6pt}{3pt}{\sl}{}{\bf}{.}{.5em}{}
\theoremstyle{theoremstyleslant}
\newtheorem{theorem}{Theorem}[section]
\newtheorem{corollary}[theorem]{Corollary}
\newtheorem{lemma}[theorem]{Lemma}
\newtheorem{proposition}[theorem]{Proposition}

\newtheorem{definition}[theorem]{Definition}

\def\defn#1{\emph{\color{blue} #1}}

\def\dist{\operatorname{dist}}
\def\diam{\operatorname{diam}}
\def\width{\operatorname{width}}
\def\Awidth{\operatorname{Awidth}}
\def\Bwidth{\operatorname{Bwidth}}
\def\prob{\operatorname{Pr}}

\def\Ln{\mathscr{L}} % generic layer family
\def\Vn{\mathscr{V}}
\def\Wn{\mathscr{W}}
\def\Xn{\mathscr{X}}
\def\Yn{\mathscr{Y}}
\def\Zn{\mathscr{Z}}

\def\L{\mathcal{L}} % generic layer family
\def\V{\mathcal{V}}
\def\W{\mathcal{W}}
\def\X{\mathcal{X}}
\def\Y{\mathcal{Y}}
\def\Z{\mathcal{Z}}

\def\A{\mathcal{A}}

\renewcommand{\qedsymbol}{$\blacksquare$}

\begin{document}

\title[Superlinear subset partition graphs]{Superlinear subset partition graphs with dimension reduction, strong adjacency, and endpoint count}

\author{Tristram C. Bogart}
\address{Departamento de Matem\'aticas, Universidad de los Andes, Carrera Primera \# 18A-12, Bogot\'a, D.C., Colombia}

\author{Edward D. Kim}
\address{Department of Mathematics, University of Wisconsin--La Crosse, 1725 State Street, La Crosse, WI 54601, USA}

\date{}


\begin{abstract}
We construct a sequence of subset partition graphs satisfying the dimension reduction, adjacency, strong adjacency, and endpoint count properties whose diameter has a superlinear asymptotic lower bound. These abstractions of polytope graphs give further evidence against the Linear Hirsch Conjecture.
\end{abstract}

\maketitle

\section{Introduction}

Asymptotic bounds on the diameters of polytopes and polyhedra are relevant in understanding the efficiency of the simplex algorithm for linear optimization. The pivot steps of the algorithm form a path of adjacent vertices along the feasibility polyhedron of the linear optimization problem. The Hirsch Conjecture asserted that the maximal diameter $\Delta(n,d)$ among $d$-dimensional polytopes with $n$ facets is at most $n-d$. (For more on the Hirsch Conjecture, refer to~\cite{Klee:The-d-step-conjecture},~\cite{Klee:d-step}, or the survey~\cite{KimSantos:HirschSurvey}.) Though Santos disproved the Hirsch Conjecture for convex polytopes in~\cite{Santos:CounterexampleHirsch}, the true growth rate of $\Delta(n,d)$ is unknown.

The best known lower bound for $\Delta(n,d)$ is $\frac{21}{20}(n-d)$ if $d \geq 20$ by Matschke, Santos, and Weibel in~\cite{MatschkeSantosWeibel:Width5Prismatoids}. The best known upper bound for $\Delta(n,d)$ as $d$ varies is the recent bound $(n-d)^{\log_2 d}$ by Todd in~\cite{Todd:KalaiKleitman}, improving on the previous upper bound $n^{1+\log_2d}$ of Kalai and Kleitman in~\cite{Kalai:Quasi-polynomial}. If the dimension $d$ is fixed, the diameter is linear in the number $n$ of facets. More specifically, if $d$ is fixed, the upper bound $\Delta(n,d) \leq \frac132^{d-3}(n-d+\frac52)$ by Barnette in~\cite{Barnette} is linear in $n$, which improved on the previous upper bounds $\Delta(n,d) \leq n2^{d-3}$ by Larman in~\cite{Larman} and $\Delta(n,d) \leq n3^{d-3}$ by Barnette in~\cite{Barnette:UpperBound}. Though $\Delta(n,d)$ is linear in $n$ when $d$ is fixed, the question of knowing whether $\Delta(n,d)$ is linear in $n$ as $d$ increases remains open. In addition to this so-called Linear Hirsch Conjecture, the Polynomial Hirsch Conjecture (which asserts that $\Delta(n,d)$ is polynomial in $n$ as $d$ increases) is also open.

In~\cite{Eisenbrand:Diameter-of-Polyhedra}, Eisenbrand et al.{} revived the study of diameters of abstractions of polytope graphs which were pioneered in the 1970s by Adler, Dantzig, Murty, and others (see~\cite{Adler:AbstractPolytopesThesis}, \cite{Adler:LowerBounds}, \cite{AdlerDantzigMurty:AbstractPolytopes}, \cite{Adler:MaxDiamAbsPoly}, \cite{AdlerSaigal:LongPathsAbstract}, \cite{Kalai:DiameterHeight}, \cite{Lawrence}, \cite{Murty:GraphAbstract}). Adler analyzed diameters of the so-called abstract polytopes. The key construction of Eisenbrand, H\"ahnle, Razborov, and Rothvo\ss{} in~\cite{Eisenbrand:Diameter-of-Polyhedra} is a connected layer family whose diameter is almost quadratic in $n$. The defining feature of a connected layer family is the connectedness property, which is inspired by the fact that the graphs of faces of a polyhedron are themselves graphs of polyhedra and thus connected.

In~\cite{Kim:PolyhedralGraphAbstractions}, Kim introduced this property to the setting of arbitrary graphs where it is called dimension reduction on a subset partition graph. Two additional properties found in the earlier abstractions are also introduced and a construction in~\cite{Kim:PolyhedralGraphAbstractions} showed that subset partition graphs satisfying these two properties have superlinear diameter. The first of these additional properties is adjacency, which is inspired by the fact that $d-1$ facets of a $d$-polytope determine an edge, thus the vertices at its endpoints should be adjacent. The second of these properties is the endpoint count property, which in terms of polytopes says that the number of vertices on an edge is two in a polytope and at most two in a polyhedron. A strengthening of adjacency is strong adjacency, which captures the fact that if two vertices of a polytope are adjacent, then they share $d-1$ facets in common.

In~\cite{Haehnle:ConstructingSPGs}, H\"ahnle gave a construction of subset partition graphs satisfying strong adjacency and endpoint count whose diameter is exponential in $n$. The main unresolved question in~\cite{Kim:PolyhedralGraphAbstractions} and H\"ahnle's discussion in~\cite{Haehnle:ConstructingSPGs} is how the diameter of subset partition graphs grows in $n$ when the dimension reduction, adjacency, and endpoint count properties are all satisfied. It would be even more interesting to ascertain that the diameter is superlinear when dimension reduction, strong adjacency, and endpoint count are all satisfied. This paper answers that question in the affirmative, starting with the superlinear construction from~\cite{Eisenbrand:Diameter-of-Polyhedra} and the doubling operation of H\"ahnle in~\cite{Haehnle:ConstructingSPGs}. The resulting subset partition graph satisfies all the desired properties and has almost quadratic diameter. This is additional evidence that $\Delta(n,d)$ may not be linear in $n$ and provides an answer to a question posed in~\cite{Eisenbrand:Diameter-of-Polyhedra},~\cite{Haehnle:ConstructingSPGs}, and~\cite{Kim:PolyhedralGraphAbstractions}.

We stress that the construction in~\cite{Eisenbrand:Diameter-of-Polyhedra} gives the \emph{first} superlinear construction with the crucial property of dimension reduction. The construction in H\"ahnle's article~\cite{Haehnle:ConstructingSPGs}, Corollary 2.12 in~\cite{Santos:RecentProgress} by Santos, and Theorem 4.4 in~\cite{Kim:PolyhedralGraphAbstractions} by Kim all give constructions with exponential diameter with the other properties but without dimension reduction. Theorem~\ref{theorem:2ndZsummary} gives the first superlinear construction with dimension reduction and all other major properties identified in the pioneering work by Adler, Dantzig, Murty et al.{} (see~\cite{Adler:AbstractPolytopesThesis}, \cite{Adler:LowerBounds}, \cite{AdlerDantzigMurty:AbstractPolytopes}, \cite{Adler:MaxDiamAbsPoly}, \cite{AdlerSaigal:LongPathsAbstract}, \cite{Kalai:DiameterHeight}, \cite{Lawrence}, \cite{Murty:GraphAbstract}). Inasmuch as subset partition graphs with all of these properties model polyhedra and polytopes, Theorem~\ref{theorem:2ndZsummary} is the best evidence thus far against the Linear Hirsch Conjecture.

In Section~\ref{section:definitions}, we introduce definitions which will be relevant throughout the paper, along with some constructions of subset partition graphs arising from the graphs of polytopes. In Section~\ref{section:Eisenbrand1}, we start with the first connected layer family of Eisenbrand et al.~in~\cite{Eisenbrand:Diameter-of-Polyhedra} to obtain a subset partition graph satisfying dimension reduction, strong adjacency, and endpoint count. Both the starting connected layer family and the resulting subset partition graph have diameter sublinear in $n$. Using similar but more elaborate constructions, in Section~\ref{section:Eisenbrand2}, we start with the second connected layer family of Eisenbrand et al.{} in~\cite{Eisenbrand:Diameter-of-Polyhedra} to obtain a subset partition graph satisfying dimension reduction, strong adjacency, and endpoint count. The resulting graph has superlinear diameter. This key result is summarized in Theorem~\ref{theorem:2ndZsummary}. Appendix~\ref{section:separation-lemmas} contains several lemmas on cardinalities of random sets that are needed in our constructions.

\section{Definitions and Preliminary Constructions}\label{section:definitions}

Terminology on polytopes as used in this manuscript can be found in~\cite{Grunbaum:Polytopes} or~\cite{Ziegler:Lectures}.
Many of the following definitions and notation are taken from~\cite{Kim:PolyhedralGraphAbstractions} and H\"ahnle's article~\cite{Haehnle:ConstructingSPGs}. 

\begin{definition}
Let $[n] = \{1,2,\dots,n\}$ be a set of \defn{symbols} and fix a set $\A \subset \binom{[n]}{d}$ of $d$-subsets. Let $\V = \{V_0,\dots,V_t\}$ be a partition of $\A$; that is, $\bigcup_{i=0}^t V_i = \A$, $V_i \neq \emptyset$ for all $i$, and $V_i \cap V_j = \emptyset$ for all $i \neq j$. A \defn{subset partition graph} $G$ is any (simple) connected graph whose vertex set is $\V$.    
\end{definition}
Given a subset partition graph $G$, we denote its diameter by $\delta(G)$. Define the \defn{Hirsch ratio} of $G$ to be~$\frac{\delta(G)}{n}$. We will be interested in the following properties that subset partition graphs may have:
\begin{itemize}
\item \textbf{dimension reduction:} If $F \subseteq [n]$ and $\left| F \right| \leq d-1$, then the induced subgraph $G_F$ on the vertices that contain a set $A$ with $A \subseteq F$ is connected.
\item \textbf{adjacency:} If $A,A' \in \A$ and $\left| A \cap A' \right| = d-1$, then $A$ and $A'$ are in the same or adjacent vertices of $G$.
\item \textbf{strong adjacency:} The graph $G$ satisfies adjacency, and in addition the following property: if $V$ and $V'$ are adjacent vertices in $G$, then there exist $A \in V$ and $A' \in V'$ such that $\left| A \cap A' \right| = d-1$. 
\item \textbf{endpoint count:} If $F \in \binom{[n]}{d-1}$, then $F$ is contained in at most two sets $A \in \A$.
\end{itemize}

We define a \defn{layer family} to be a subset partition graph whose underlying graph is a path and a \defn{layer} to be the collection of sets associated to an individual vertex of the path. A \defn{connected layer family} is a layer family that satisfies dimension reduction.

Following~\cite{GoreskiMacPherson} and~\cite{NovikSwartz}, a pure simplicial complex $\Delta$ is a pseudomanifold if each codimension-one face of $\Delta$ is contained in exactly two facets. In particular, there is no assumption that $\Delta$ is strongly connected.
In~\cite{Santos:RecentProgress}, Santos identifies the pseudomanifold property as a pure simplicial complex where each codimension-one face is contained in at most two facets. Under this definition, a subset partition graph whose set $\A$ satisfies endpoint count (such as our final construction in Theorem~\ref{theorem:2ndZsummary}) is a pseudomanifold.
As Santos indicates in~\cite{Santos:RecentProgress}, it is interesting to ascertain if $\Delta$ is normal, that is, is the star of every simplex in $\Delta$ connected. Though we suspect the complex $\A$ in our construction for Theorem~\ref{theorem:2ndZsummary} is not normal, this remains open.

The remainder of this section presents a collection of constructions of subset partition graphs related to the graphs of polytopes. Though these statements are not referenced in the subsequent sections, the constructions provide polytopal motivation for studying subset partition graphs. We start by recalling a remark from~\cite{Kim:PolyhedralGraphAbstractions} which will be needed later:

\begin{proposition}[\cite{Kim:PolyhedralGraphAbstractions}]\label{proposition:propertypreservation} The dimension reduction, adjacency, and endpoint count properties are all preserved under edge addition and edge contraction.\hfill\qedsymbol
\end{proposition} 

Trivially, the property that adjacent vertices contain $d$-sets intersecting in $d-1$ elements is preserved under edge contraction but not edge addition. In contrast, the property that each vertex has precisely one $d$-set is preserved under edge addition but not under edge contraction. 

The motivation for these properties is to study graphs of polytopes. Let $P$ be a simple $d$-polytope whose facets are indexed by $[n]$. Define two subset partition graphs as follows: 

\begin{enumerate}
\item For each vertex $v$ of $P$, define the set 
\[A_v = \{i \in [n]: \text{ facet } i \text{ contains } v\}\]
and let $V_v = \{A_v\}$. Define a graph $G_P$ with vertex set $\{V_v\}$ and edges $\{V_v, V_w\}$ whenever $\{v,w\}$ is an edge of $P$. That is, if we forget the subsets then $G_P$ is isomorphic to the graph of $P$ itself. This is the simplicial complex dual to $P$, with facets of the form $A_v$.
\item Fix a vertex $u$ of $P$. For each vertex $v$ of $P$, define $A_v$ as above. Let $V_i = \{A_v: \dist(u,v) = i\}$. Define $G_P^u$ to be the path whose vertices are (in order) $V_0,V_1,\dots,V_k$, where $k$ is the largest distance from $u$ to any vertex of $P$. This subset partition graph motivates the connected layer families defined by Eisenbrand et al.{} in~\cite{Eisenbrand:Diameter-of-Polyhedra}.
\end{enumerate} 

We observe that $G_P$ has the dimension reduction, strong adjacency, and endpoint count properties, vertex consists of exactly one $d$-set, and that its diameter equals the diameter of $P$. In particular, $G_P$ is a pseudomanifold in the sense of~\cite{Santos:RecentProgress}. Furthermore, we can obtain $G_P^u$ from $G_P$ as follows. First, if $v,w$ are two vertices such that $\dist(u,v) = \dist(u,w)$ and such that $v, w$ are not already adjacent, then add the edge $\{v,w\}$. Second, contract all edges $\{v,w\}$ for which $\dist(u,v) = \dist(u,w)$. Thus $G_P^u$ also has dimension reduction, adjacency, and endpoint count by Proposition~\ref{proposition:propertypreservation}. In addition to adjacency, the subset partition graph $G_P^u$ also satisfies strong adjacency because the added edges that might violate this property are all contracted away. Obviously, $G_P^u$ fails the property of unit vertex cardinality on nearly every vertex. Its diameter may be strictly less than the diameter of $G$, but the diameters will be equal if $u$ is chosen appropriately. 

This construction works more generally as follows, and thus our constructions in the following sections focus on layer families:
\begin{lemma} 
Let $G$ be a subset partition graph with symbol set $[n]$ and dimension $d$. Then there exists a subset partition graph $G'$ with symbol set $[n]$ and dimension $d$ satisfying the following.
\begin{enumerate}
\item The underlying graph of $G'$ is a path. 
\item The diameter (edge length) of $G'$ equals the diameter of $G$. 
\item If $G$ satisfies any or all of the dimension reduction, adjacency, strong adjacency, and endpoint count properties, then $G'$ satisfies the same properties. 
\end{enumerate} 
\label{lemma:topath}
\end{lemma}

\begin{proof}
Choose a vertex $u$ of $G$ whose distance to some other vertex equals $\delta(G)$. For each pair of nonadjacent vertices $v, w$ such that $\dist(u,v)$ = $\dist(u,w)$, add an edge from $v$ to $w$. This does not change the distance of any vertex from $u$; in particular, the diameter still equals $\delta(G)$. For all $0 \leq k \leq \delta(G)$, the set $V_k$ of vertices at distance $k$ from $u$ is now a clique. Contract edges to turn this clique into a single vertex $\V_k$. The result is a path of length $\delta(G)$. If $G$ has dimension reduction, adjacency, and/or endpoint count, then so does $G'$ by~\ref{proposition:propertypreservation}.
\end{proof}

Alternatively, we can preserve most of the properties while splitting vertices instead of merging them. The operation is not a true inverse to the previous one: applying it to the connected layer family of the graph of a polytope does not produce the original graph, but instead yields a graph with many extra edges. In addition, the operation does not in general preserve strong adjacency.

\begin{lemma}
Let $G$ be a subset partition graph with symbol set $[n]$ and dimension $d$. Then there exists a subset partition graph $G''$ with symbol set $[n]$ and dimension $d$ satisfying the following.
\begin{enumerate}
\item Each vertex $V \in \V$ of $G''$ consists of exactly one $d$-set.
\item The diameter of $G''$ equals the diameter of $G$. 
\item If $G$ satisfies any or all of the dimension reduction, adjacency, and endpoint count properties, then $G''$ satisfies the same properties. 
\end{enumerate} 
\label{lemma:tosingleton}
\end{lemma}

\begin{proof}
To form $G''$, replace each vertex $V$ of $G$ by a clique of vertices $W_{V,1}, \dots W_{V,\#(V)}$ that each contain a single set $A \in V$. For each edge $\{U,V\}$ in $G$, connect $W_{V,i}$ to $W_{U,j}$ for all $i$ and $j$. Now $\dist(W_{V,i}, W_{V,j}) = 1$ for any vertex $V$ and any indices not equal to $j$, and $\dist(W_{U,i},W_{V,j})$ in $G''$ equals $\dist(U,V)$ in $G$ for any vertices $U \neq V$ and any indices $i$ and $j$. So as long as $G$ has at least one edge, we have $\diam(G'') = \diam(G)$. 

We now show this construction preserves each of the three key properties. 
\begin{enumerate}
\item Suppose $G$ has dimension reduction. Let $F \subseteq [n]$. The vertices $V_1, \dots, V_k$ that contain a set containing $F$ form a connected subgraph $H$ of $G$. Then the vertices of $G''$ whose set contains $F$ are of the form $W_{V_i,j}$ with at least one valid choice of $j$ for each $1 \leq i \leq k$. These vertices induce a subgraph $H''$ of $G''$ whose structure is obtained by replacing each vertex of $H$ by a clique of size at least one. So $H''$ is connected. 
\item Suppose $G$ has adjacency. Let $F \subseteq [n]$ with $\left| F \right| = d-1$. Suppose $A$ and $A'$ are $d$-sets that contain $F$ and appear in $G''$. Then $A$ and $A'$ also occur in $G$, and by hypothesis appear in the same or adjacent vertices of $G$. Then by construction, $A$ and $A'$ appear in adjacent vertices of $G''$.   
\item Suppose $G$ has endpoint count. Since the sets $A$ that appear in $G''$ are the same as those that appear in $G$, $G''$ also has endpoint count. 
\end{enumerate}  
Of course, the reason that strong adjacency is not preserved above is that a subset partition graph with strong adjacency where each vertex has cardinality one must be the graph $G$ itself. 
\end{proof}

Since both these constructions preserve the number of symbols and the diameter, they preserve the Hirsch ratio. Thus in constructing graphs that satisfy adjacency, endpoint count, and dimension reduction, we may at any point choose to work either with paths or with graphs satisfying the property that each vertex contains precisely one set. In the following sections we will work only with paths.  

\section{Modifying the first Eisenbrand et al. Construction}\label{section:Eisenbrand1}

In this section, we construct a layer family with all desired properties based on the first construction in~\cite{Eisenbrand:Diameter-of-Polyhedra}. Our construction will define layer families of dimension $d$ and of dimension $2d$ at the same time: layer families of dimension $d$ will be defined on the symbol set $[n]$ while layer families of dimension $2d$ will be defined on the symbol set $[2] \times [n]$. We can represent subsets of $[2] \times [n]$ as two-row diagrams of boxes as in~\cite{Haehnle:ConstructingSPGs}. As much as it is possible, we use $C$ and $E$ to denote subsets of $[n]$, and use $S$ and $T$ to denote subsets of $[2] \times [n]$. Given a set $C$, we define the doubling operation $D : 2^{[n]} \rightarrow 2^{[2] \times [n]}$ as $D(C) = [2] \times C$. Define the vertical projection $\tilde{\pi}: [2] \times [n] \to [n]$ by $\pi((1,i)) = \pi((2,i)) = i$, and as a pseudo-inverse to $D$, define $\pi$ on a subset $S \subseteq [2] \times [n]$ to be the direct image of $S$ under $\tilde{\pi}$. For $S \subseteq [2] \times [n]$, we define $\width(S) = \# (\pi(S))$. 

A set $C \subseteq [n]$ is \defn{active} on layer $\L_k$ if there exists a $d$-set $E \in \L_k$ with $C \subseteq E$. We define the following  properties for a layer family $\Ln = (\L_1,\dots,\L_\ell)$, to be used here and in Section~\ref{section:Eisenbrand2}.  
\begin{itemize} 
\item \textbf{covering:} Let $C$ be a subset of $[n]$ with $\left| C \right| \leq d-1$. Then $C$ is active on every layer of $\Ln$.
\item \textbf{$m$-covering:} Let $C$ be a subset of $[n]$ of size at most $d-1$. Then for every layer $\L_k$ of $\Ln$, there exist at least $m$ different $d$-sets $E_1, \dots, E_m \in \L_k$ such that $C \subseteq E_i$.
\item \textbf{completeness:} If $C$ is a $d$-subset of $[n]$, then $C \in \L_k$ for some $k$. 
\item \textbf{linkage:} For every pair of adjacent layers $\L_k, \L_{k+1}$, there exist $A \in \L_k$ and $A' \in \L_{k+1}$ such that $\left| A \cap A' \right| = d-1$. 
\end{itemize}     
Note that covering is equivalent to $1$-covering. Also $m$-covering (for any $m \geq 1$) implies covering and covering implies dimension reduction. Linkage together with adjacency is equivalent to strong adjacency for a layer family: the new notion is introduced because an intermediate step in our construction is a layer family with linkage that does not satisfy adjacency.

In~\cite{Eisenbrand:Diameter-of-Polyhedra}, Eisenbrand, H{\"a}hnle, Razborov, and Rothvo{\ss} present two constructions of connected layer families. The first construction produces a connected layer family $\Vn=(\V_1, \V_2, \dots, \V_\ell)$ on $n$ symbols of any dimension $d \leq n$ that satisfies covering and completeness, and hence also dimension reduction.

Let $\Lambda(n,d)$ denote the maximal number of layers in a $d$-dimensional connected layer family on $n$ symbols. 
Eisenbrand et al.{} establish (see Theorem 4.1 in~\cite{Eisenbrand:Diameter-of-Polyhedra}) that 
\[\Lambda(n,d) \geq \left\lfloor \frac{n - (d-1)}{\ln n} \right\rfloor\]
by proving the existence of many disjoint $(n,d,d-1)$-covering designs. Each covering design is a layer in their first construction of a connected layer family.

Our goal in this section is to modify this construction to preserve dimension reduction and add the properties of endpoint count and strong adjacency while not reducing the diameter by more than a constant factor. In the next section, we do the same to the second construction of Eisenbrand et al., which has superlinear diameter.

The steps are as follows.
\begin{enumerate}
\item From $\Vn$, construct a new family $\Wn$ by merging groups of $m$ consecutive layers. (Eventually we will need to choose any $m \geq 3$.) This strengthens covering to $m$-covering. In particular, $\Wn$ still satisfies dimension reduction. Let $\delta$ be the number of layers of $\Vn$, which is roughly $1/m$ times the number of layers of $\Vn$. 
\item From $\Wn$, construct a new family $\Xn$ by applying the map $D$ that duplicates each symbol. Then $\Xn$ has adjacency and endpoint count for a trivial reason: each $(2d-1)$-set is a subset of at most one $2d$-set in the entire family. Furthermore, although $\Xn$ no longer satisfies completeness, it still satisfies dimension reduction for the following reason. For any $S \subseteq [2] \times [n]$ of width at most $d-1$ and for any $1 \leq k \leq \delta$, there exists $T$ in the $k$th layer of $\Vn$ that contains $\pi(S)$. Then $D(T)$ appears in the $k$th layer of $\Xn$ and $D(T)$ contains $S$. So $S$ is active on every layer of $\Xn$. Wider sets are active on at most one layer. 
\item From $\Xn$, construct a new family $\Yn$ by adding one ``interpolation set'' to each layer except the last. The interpolation sets are chosen nearly randomly so that they will have certain properties with high probability. They are also chosen so that $\Yn$ has linkage. The new sets cause $\Yn$ not to have dimension reduction or adjacency, though it still does have endpoint count.      
\item From $\Yn$, construct a new family $\Zn$ by removing a small number of correspondence sets that interact with the interpolation sets to cause the problems with dimension reduction and adjacency in $\Yn$. We show that $\Zn$ has dimension reduction, strong adjacency, and endpoint count. 
\end{enumerate}  

\subsection{Obtaining $m$-covering}

We begin by describing the construction of $\Wn$ from $\Vn$. Fix $m \geq 1$. If $\ell$ is the number of layers of the layer family $\Vn$, write $\ell = \delta m + r$ with $0 \leq r < m$. Construct a new family $\Wn = (\W_1, \W_1, \dots, \W_\delta)$ by merging consecutive groups of $m$ layers. For $\W_\delta$, merge the last $m+r$ layers. That is, $\W_1 = \V_1 \cup \V_2 \cup \dots \cup \V_{m}$, $\W_2 = \V_{m+1} \cup \dots \cup \V_{2m}$, \dots, $\W_\delta = \V_{(\delta-1)m+1} \cup \dots \cup \V_{\delta m + r}$.  
Thus the diameter $\delta=\delta(\Wn)$ of the new family $\Wn$ satisfies $\delta+1 = \lfloor \ell/m\rfloor$. Also, $\Wn$ satisfies $m$-covering because $\Vn$ satisfies covering. The layer family $\Wn$ is complete because $\Vn$ is complete and the two families contain exactly the same $d$-sets. 

Note that if $m=1$, then $\Wn=\Vn$.  

\begin{theorem}
The layer family $\Wn$ satisfies $m$-covering and thus dimension reduction. The diameter of $\Wn$ is in $\Omega(n/\log n)$.\hfill\qedsymbol
\end{theorem}

\subsection{Doubling}

For the next step, we perform a doubling operation inspired by H\"ahnle's construction in~\cite{Haehnle:ConstructingSPGs}. Construct a new family $\Xn=(\X_1,\X_2,\dots,\X_\delta)$ of dimension $2d$ on the symbol set $[2] \times [n]$ of size $2n$ by duplicating a symbol every time it occurs in $\Wn$. That is, for each $k = 1,\dots,\delta$, define  
\[\X_k = \{ D(C) : C \in \W_k\}.\]

Every $2d$-set in $\Xn$ consists of two identical rows, so its width is exactly $d$. We call these $2d$-sets \defn{correspondence sets}: each set $D(C)$ in $\X_k$ corresponds to the set $C \in \W_k$. 

\begin{theorem}
The layer family $\Xn$ satisfies adjacency, endpoint count, $m$-covering, and dimension reduction. The diameter of $\Xn$ is in $\Omega(n/\log n)$.\hfill\qedsymbol
\end{theorem}

Since no $(2d-1)$-set is active on more than one layer, $\Xn$ fails strong adjacency at every pair of consecutive layers. Our next step is to rectify this. 

\subsection{Obtaining linkage}

We now modify $\Xn$ to construct a new family $\Yn=(\Y_1,\dots,\Y_\delta)$ by adding a $2d$-set to each layer so that $\Yn$ has linkage. Choose (without replacement) random $(d-1)$-subsets $P_1,\dots,P_{\delta-1}$ of $[n]$. Assume that $d$ and $n-d$ are sufficiently large so that Proposition~\ref{proposition:P-intersection} allows us to stipulate that among the $P_1, P_2, \dots, P_{\delta-1}$, each pairwise intersection is of size at most $d-4$. By the covering property for $\Wn$, for each $k$ there exists a $d$-set $G_k \in \W_{k+1}$ that contains $P_k$. Write $G_k = P_k \cup \{g_k\}$. See Figure~\ref{figure:construction1-interpolation}. Choose $f_k \notin G_k$ such that $F_k := P_k \cup \{f_k\} \in \W_k$. Let $I_k$ be the $2d$-set whose first row is $F_k$ and whose second row is $G_k$. Also let $U_k = P_k \cup \{f_k,g_k\}$, so we have $\pi(I_k) = U_k$. The sets of the form $I_k$ are called \defn{interpolation sets}.

\begin{figure}[hbt]
\begin{tikzpicture}
\draw [decorate,decoration={brace,amplitude=4pt},xshift=0pt,yshift=0pt]
(1.25,1.3) -- (3.15,1.3) node [black,midway,yshift=0.4cm]
{\footnotesize $d-1$};
\draw [decorate,decoration={brace,amplitude=4pt},xshift=-4pt,yshift=0pt]
(-.1,0) -- (-.1,1.2) node [black,midway,xshift=-0.6cm] 
{$I_{k}$};
\draw (.6,.6) rectangle (1.2,1.2);
\node at (.9,.9) {$f_{k}$};
\draw (0,0) rectangle (.6,.6);
\node at (.3,.3) {$g_{k}$};
\draw (1.2,1.2) rectangle (3.2,.6);
\node at (2.2,.9) {$P_{k}$};
\draw (1.2,.6) rectangle (3.2,0);
\node at (2.2,.3) {$P_{k}$};
\node at (4.1,.9) {$F_k\in \W_{k}$};
\node at (4.32,.3) {$G_k \in \W_{k+1}$};
\draw[->] (1.8,-.2) --node[anchor=west]{$\pi$} (1.8,-.6);
\draw (0,-1.5) rectangle (3.2,-.9);
\node at (1.6,-1.2) {$U_{k}$};
\end{tikzpicture}
\caption{Schematic diagram of an interpolation set $I_{k}$ and related sets}
\label{figure:construction1-interpolation}
\end{figure}
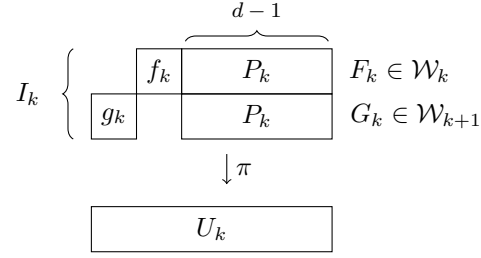

Now for $1 \leq k \leq \delta-1$, define the layer $\Y_k$ of the new family by $\Y_k = \X_k \cup \{I_k\}$. Finally, define $\Y_\delta = \X_\delta$. 

The idea is that $I_k$ has $2d-1$ elements in common with $[2] \times F_k \in \Y_k$ and $2d-1$ elements in common with $[2] \times G_k \in \Y_{k+1}$. In particular, the latter fact gives us the property that layers $\Y_k$ and $\Y_{k+1}$ contain (respectively) the $2d$-sets $I_k$ and $[2] \times G_k$.

\begin{corollary} \label{corollary:Uintersection}
For distinct indices $k$ and $k'$, $\left| U_k \cap U_{k'} \right| \leq d-2$. 
\end{corollary}

\begin{proof}
By construction, $P_k$ and $P_{k'}$ have at most $d-4$ elements in common. The result follows since $U_k$ and $U_{k'}$ are obtained by adding two elements to (respectively) $P_k$ and $P_{k'}$. 
\end{proof}

\begin{theorem}
The layer family $\Yn$ has endpoint count and linkage. The diameter of $\Yn$ is in $\Omega(n/\log n)$.\hfill\qedsymbol
\end{theorem}

\subsection{Obtaining adjacency and dimension reduction}

We need to make one more modification in order to create adjacency (and thus strong adjacency) and dimension reduction. For each interpolation set $I_k$, we will remove certain correspondence sets that combine with $I_k$ to violate these properties. For each $k$ and for each $a \in P_k$, define 
\[R_k^a = D\left( (P_k \setminus \{a\}) \cup \{f_k,g_k\} \right).\]

That is, $R_k^a$ is a correspondence set such that $\pi(R_k^a) \subseteq U_k$ but $R_k^a \neq D(F_k)$ and $R_k^a \neq D(G_k)$. See Figure~\ref{figure:Rka}. The idea is that $R_k^a$ may combine with $I_{k}$ to create a violation of dimension reduction in $\Yn$. Specifically, $R_k^a \cap I_k$ is a set of width $d$ that is active on at least two layers: layer $k$ and the layer containing $R_k^a$. These two layers may not be adjacent, and there is nothing to guarantee that $R_k^a \cap I_k$ is active on the layers in between the two. 

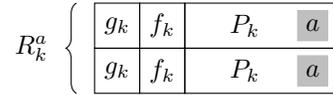
\begin{figure}[hbt]
\begin{tikzpicture}
\draw [decorate,decoration={brace,amplitude=4pt},xshift=-4pt,yshift=0pt]
(-.1,0) -- (-.1,1.2) node [black,midway,xshift=-0.6cm] 
{$R_{k}^a$};
\draw (.6,.6) rectangle (1.2,1.2);
\node at (.9,.9) {$f_{k}$};
\draw (.6,0) rectangle (1.2,0.6);
\node at (.9,.3) {$f_{k}$};
\draw (0,0) rectangle (.6,.6);
\node at (.3,.3) {$g_{k}$};
\draw (0,.6) rectangle (.6,1.2);
\node at (.3,.9) {$g_{k}$};
\draw (1.2,1.2) rectangle (3.2,.6);
\node at (2,.9) {$P_{k}$};
\draw (1.2,.6) rectangle (3.2,0);
\node at (2,.3) {$P_{k}$};
\draw[fill=lightgray, color=lightgray] (2.7,.7) rectangle (3.1,1.1);
\draw[fill=lightgray, color=lightgray] (2.7,.1) rectangle (3.1,0.5);
\node at (2.9,.3) {$a$};
\node at (2.9,.9) {$a$};
\end{tikzpicture}
\caption{Schematic diagram of the correspondence $2d$-set $R_{k}^a$. The shaded element $a \in P_{k}$ is removed: $(a,1),(a,2)\not\in R_k^a$.}
\label{figure:Rka}
\end{figure}

We thus define the collection  
$$\mathcal{R} = \{R_k^a: \; 1 \leq k \leq \delta, a \in P_k \}$$  
and define a new family 
\[\Zn = (\Z_1, \Z_2, \dots \Z_\delta)\]
where for each $k$,  
\[\Z_k = \Y_k \setminus \mathcal{R}.\] 

We will proceed to show that the family $\Zn$ has dimension reduction and all of the other desired properties. 

\begin{lemma}\label{lemma:widthclassificationinZ}
Let $S \subseteq [2] \times [n]$. 
\begin{enumerate}
\item If $\width(S) \geq d+2$, then no set in $\Zn$ contains $S$. 
\item If $\width(S) = d+1$, then no correspondence set and at most one interpolation set in $\Zn$ contains $S$. 
\item If $\width(S) = d$, then at most one correspondence set and at most one interpolation set in $\Zn$ contain $S$. 
\end{enumerate} 
\end{lemma}
      
\begin{proof} $ $
\begin{enumerate}
\item Suppose $\width(S) \geq d+2$. Each correspondence set has width $d$ and each interpolation set has width $d+1$, so none contain $S$. 
\item Suppose $\width(S) = d+1$. Then again no correspondence set is wide enough to contain $S$. If $I_k$ and $I_k'$ are distinct interpolation sets that both contain $S$, then by applying $\pi$, we see that $U_k$ and $U_k'$ both contain $\pi(S)$. This implies $\left| U_k \cap U_k' \right| \geq \left| \pi(S) \right| = d+1$, contradicting Corollary~\ref{corollary:Uintersection}. 
\item Suppose $\width(S) = d$. If $T$ is a correspondence set that contains $S$, then $\pi(T) \supseteq \pi(S)$. But both $\pi(T)$ and $\pi(S)$ have width $d$, so in fact $\pi(T) = \pi(S)$ and thus by the form of correspondence sets, $T = D(\pi(S))$. Thus no more than one correspondence set contains $S$. 

If $I_k$ and $I_k'$ are distinct interpolation sets that both contain $S$, then by the same argument as in the previous statement, we see that $\left| U_k \cap U_k' \right| \geq \left| \pi(S) \right| = d$, again contradicting Corollary~\ref{corollary:Uintersection}. So at most one interpolation set contains $S$.  
\end{enumerate}
\end{proof}

\begin{proposition} \label{proposition:endpointcountinZ}
The layer family $\Zn$ has endpoint count. 
\end{proposition}

\begin{proof}
Let $S \subseteq [2] \times [n]$ be a set of size $2d-1$. Then one row of $S$ has cardinality at least $d$, so $\width(S) \geq d$. By considering each of the three cases of Lemma~\ref{lemma:widthclassificationinZ}, we see that at most two sets in $\Zn$ contain $S$, verifying endpoint count. 
\end{proof}

\begin{lemma} \label{lemma:adjacentlayerZ}
Let $S$ be a subset of $[2] \times [n]$ of width $d$. Then either $S$ is contained in at most one $2d$-set of $\Zn$, or it is contained in exactly two $2d$-sets of $\Zn$ that occur in the same or adjacent layers.  
\end{lemma}

\begin{proof}
If $S$ is contained in at most one set in $\Zn$, then the conclusion is immediate. By Lemma~\ref{lemma:widthclassificationinZ}, the only other possibility is that $S$ is contained in one correspondence set $T$ and one interpolation set $I_k$. Furthermore, since $\width(S) = d$, we know that $\pi(S) = \pi(T)$. Since $S \subset I_k$, we can apply $\pi$ to get $\pi(T) = \pi(S) \subseteq U_k$. So $\pi(T) = U_k \setminus \{a\}$ for some $a \in [n]$.  

If $a \notin \{f_k,g_k\}$, then we have $\pi(T) = R_k^a$, but this is a contradiction because all such sets were removed in the construction of $\Zn$. If $a = f_k$, then $T=P_k \cup \{g_k\} = G_k$. Then the two sets containing $S$ are $I_k \in \Z_k$ and $D(G_k) \in \Z_{k+1}$. If $a=g_k$, then $\pi(T)=P_k \cup \{f_k\} = F_k$. Then the two sets containing $S$ are $I_k$ and $D(F_k)$ which are both in layer $\Z_k$.
\end{proof}

\begin{proposition} \label{proposition:adjacencyinZ}
The layer family $\Zn$ has adjacency.
\end{proposition}

\begin{proof}
Let $S$ be a $(2d-1)$-subset of $[2] \times [n]$. As before, this implies that $\width(S) \geq d$. If $\width(S) \geq d+2$ or $\width(S) = d+1$, then the first or the second cases of Lemma~\ref{lemma:widthclassificationinZ} imply that $S$ is contained in at most one set in $\Zn$. Adjacency immediately follows. Finally, if $\width(S) = d$, then Lemma~\ref{lemma:adjacentlayerZ} implies that $S$ is active on only one layer or on two adjacent layers of $\Zn$, so again adjacency follows.
\end{proof}

\begin{proposition} \label{proposition:linkageinZ}
The layer family $\Zn$ has strong adjacency. 
\end{proposition}

\begin{proof}
For each $k \in \{1,\dots,\delta-1\}$, the sets $I_k \in \Y_k$ and $D(G_k) \in \Y_{k+1}$ provide strong adjacency between these layers in $N$. So to show strong adjacency in $\Zn$, it suffices to show that none of these sets are removed in the construction of $\Zn$. By definition, no interpolation set is removed, so we only need to consider the sets $D(G_k)$. 

Suppose $D(G_k)$ is removed in the construction of $\Zn$. Then $D(G_k) = D(R_{k'}^a)$ for some ${k'} \neq k$ and some $a \notin U_{k'}$. But
\begin{align*}
D(G_k) = D(R_{k'}^a) & \Leftrightarrow  G_k = R_{k'}^a \\
 & \Leftrightarrow  P_k \cup \{g_k\} = \left(P_{k'} \cup \{f_k,g_k\}\right) \setminus \{a\} \\
 & \Rightarrow  P_k \subset P_{k'} \cup \{f_{k'},g_{k'}\} \\
& \Rightarrow  \left| P_k \cap P_{k'} \right| \geq d-3
\end{align*}
and this last statement contradicts Proposition~\ref{proposition:P-intersection}.
\end{proof}

\begin{proposition} \label{proposition:widthcoveringinZ}
If we choose $m \geq 3$, then the layer family $\Zn$ has width-covering.  
\end{proposition}

\begin{proof}
We need to show that for each layer $k$ and for each $S \subseteq [2] \times [n]$ of width at most $d-1$, the set $S$ is active on $\Z_k$. In fact we will show that there is a correspondence set $D(T) \in \Z_k$ that contains $S$. Note that to show that a correspondence set $D(T)$ belongs to $\Z_k$, it is sufficient to show that 
\[D(T) \in \Y_k \textup{ and } \forall \, {k'}=1,\dots,\delta:\, D(T) \nsubseteq D(U_k),\]
which in turn is equivalent to the condition 
\[T \in \W_k \textup{ and } \forall \, {k'}=1,\dots,\delta:\, T \nsubseteq U_{k'}.\]

Fix $k$ and $S$. Let $\widetilde{S}$ be a $(d-1)$-subset of $[n]$ that contains $\pi(S)$. By Corollary~\ref{corollary:Uintersection}, we see that $\widetilde{S}$ is contained in $U_{k'}$ for at most one value of ${k'}$.

First, suppose that $\widetilde{S}$ is not contained in any $U_{k'}$. By the $m$-covering property for $\Wn$, there are at least $m$ sets $T_1,\dots,T_m$ in $\W_k$ that contain $\widetilde{S}$. By assumption, none of these sets are contained in any $U_{k'}$, so we can take any $D(T_i)$ as the needed correspondence set in $\Z_k$.  

Now suppose that $\widetilde{S} \subseteq U_{k'}$ for a unique value of ${k'}$. Write $U_{k'} = \widetilde{S} \cup \{ s,t \}$. Again, there are at least $m$ $d$-sets $T_1, \dots, T_m$ in $\W_k$ that contain $\widetilde{S}$. However, there are only two $d$-subsets of $[n]$ that contain $\widetilde{S}$ and are contained in $U_{k'}$: $\widetilde{S} \cup \{s\}$ and $\widetilde{S} \cup \{t\}$. Since $m \geq 3$, there is at least one value of $i \in \{1,\dots,m\}$ such that $\widetilde{S} \subseteq T_i \nsubseteq U_{k'}$. For any such $i$, take $D(T_i)$ as the needed correspondence set in $\Z_k$.  
\end{proof}
 
\begin{theorem}\label{theorem:dimensionreductioninZ}
The family $\Zn$ has dimension reduction. 
\end{theorem}

\begin{proof}
Let $S \subseteq [2] \times [n]$ with $\left| S \right| \leq 2d-1$. If $\width(S) \leq d-1$, then width-covering implies that every layer of $\Zn$ includes a $2d$-set containing $S$. If $\width(S)=d$, then Lemma~\ref{lemma:adjacentlayerZ} implies that the layers on which $S$ are active are consecutive.  Finally, if $\width(S) = d+1$ or $\width(S) \geq d+2$, then again by the first two cases of Lemma~\ref{lemma:widthclassificationinZ}, at most one set in $\Zn$ contains $S$. In all cases, the set of layers that contain $S$ is an interval. 
 \end{proof}

The layer family $\Zn$ has width-covering, dimension reduction, strong adjacency, and endpoint count. Its diameter is $\frac{1}{m}$ times the diameter of the original connected layer family $\V$ and its number of symbols is twice as large. So the Hirsch ratio of $\Zn$ is within a constant factor of the Hirsch ratio of $\Vn$ which is $\Omega\left( \frac{1}{\log n}\right)$.  

\section{Modifying the second Eisenbrand et al. Construction}\label{section:Eisenbrand2}

In this section, we start with another construction of Eisenbrand et al.~from~\cite{Eisenbrand:Diameter-of-Polyhedra} and construct subset partition graphs satisfying strong adjacency, endpoint-count, and dimension reduction with superlinear diameter. Recall that $\Lambda(n,d)$ denotes the maximal number of layers in a $d$-dimensional connected layer family on $n$ symbols. 

As in Section~\ref{section:Eisenbrand1}, we will consider layer families of dimension $d$ and of dimension $2d$ at the same time.
Layer families of dimension $d$ will be defined on the symbol set $\mathbf{A} \cup \mathbf{B}$ on two disjoint sets of cardinality $\frac{n}{2}$ each, while layer families of dimension $2d$ will be defined on the symbol set $[2] \times (\mathbf{A} \cup \mathbf{B})$. Our two-row diagrams will now be partitioned with elements of $\mathbf{A}$ on the left and elements of $\mathbf{B}$ on the right. As much as it is possible, we use $C$ and $E$ to denote subsets of $\mathbf{A} \cup \mathbf{B}$, $S$ and $T$ to denote subsets of $[2] \times (\mathbf{A} \cup \mathbf{B})$. The doubling operation $D$ and the vertical projection $\pi$ are defined analogously.

We define the notion of a \defn{sectioned layer family}. Given a layer family $\Ln$, a \defn{sectioning} of $\Ln$ is a function from its layers to the non-negative integers. All layers assigned the same integer $j$ are in \defn{section} $j$.

The second construction of Eisenbrand et al.~(see~\cite{Eisenbrand:Diameter-of-Polyhedra}) is a sequence of connected layer families (indexed by an increasing $n$, which is always a multiple of $4$). For infinitely many $n$, they construct a connected layer family $\Vn$ on $n$ symbols of dimension $d=\frac{n}{4}$.
The construction starts with two disjoint symbol sets $\mathbf{A}$ and $\mathbf{B}$ of size $n'=n/2$ each. We call a subset $C$ of $\mathbf{A} \cup \mathbf{B}$ an $(a,b)$-set if it contains exactly $a$ elements of $\mathbf{A}$ and $b$ elements of $\mathbf{B}$. For $C \subseteq \mathbf{A} \cup \mathbf{B}$, we define the $\mathbf{A}$-part and $\mathbf{B}$-part of $C$ to be $C \cap \mathbf{A}$ and $C \cap \mathbf{B}$, respectively. Finally, for sets $S \subseteq [2] \times (\mathbf{A} \cup \mathbf{B})$, define $\width(S) = \left| \pi(S) \right|$, $\Awidth(S) = \left| \pi(S) \cap \mathbf{A} \right|$, and $\Bwidth(S) = \left| \pi(S) \cap \mathbf{B} \right|$.  

We will need to extend the idea of covering to a bigraded notion that applies to sectioned layer families. Given a layer family 
\[\Ln = (\L_{1,1},\L_{1,2},\dots,\L_{1,\ell_j},\L_{2,1},\dots,\L_{2,\ell_2},\dots,\L_{d-1,\ell_{d-1}}),\]
we define the following properties:
\begin{itemize} 
\item \textbf{sectioned-covering:} 
The layer family $\Ln$ has sectioned-covering if for each subset $C \subseteq \mathbf{A} \cup \mathbf{B}$ with $|C| \leq d-1$, if $|C \cap \mathbf{A}| \leq d-j$ and $|C \cap \mathbf{B}| \leq j$, then every layer of the form $\L_{j,k}$ has a $d$-set $E$ such that $C \subseteq E$. 

\item \textbf{$m$-sectioned-covering:} 
Fix a positive integer $m$. The layer family $\Ln$ has $m$-sectioned-covering if for each subset $C \subseteq \mathbf{A} \cup \mathbf{B}$ with $|C| \leq d-1$, if $|C \cap \mathbf{A}| \leq d-j$ and $|C \cap \mathbf{B}| \leq j$, then every layer of the form $\L_{j,k}$ has a $m$ distinct $d$-sets $E_1,\dots,E_m$ such that $C \subseteq E_i$. (The $m$-sectioned-covering property, or \defn{multiple sectioned-covering property}, implies sectioned-covering which implies dimension reduction. For the latter implication, see Lemma 4.2 in ~\cite{Eisenbrand:Diameter-of-Polyhedra}.)
\end{itemize}

The layers $\V_{1,1}, \dots, \V_{1,\ell_1}, \dots, \V_{d-1,\ell_{d-1}}$ are obtained by stacking $d-1$ \defn{meshes} (defined below) together.
The layers $\V_{j,1}, \dots, \V_{j,\ell_j}$ are the layers in the $j$th mesh $\mathscr{M}(\mathbf{A},d-j; \mathbf{B}, j)$, each of which consists of a collection of $(d-j,j)$-sets.
 
The number $\ell_j$ of layers in the $j$th mesh satisfies the following inequality:

\begin{proposition}[\cite{Eisenbrand:Diameter-of-Polyhedra}]\label{proposition:2nd-layers-in-mesh}
Let $n$ be a multiple of $4$ and $\mathbf{A},\mathbf{B}$ disjoint sets of size $n/2$. The number $\ell_j$ of layers in the $j$th mesh $\mathscr{M}(\mathbf{A},d-j; \mathbf{B}, j)$ satisfies
\begin{equation}
\ell_j \geq \min\left\{
\left\lfloor \frac{\frac{n}{2}-(d-j-1)}{3 \ln \frac{n}{2}} \right\rfloor
,
\left\lfloor \frac{\frac{n}{2}-(j-1)}{3 \ln \frac{n}{2}} \right\rfloor
\right\}.
\end{equation}
\end{proposition}
\begin{proof}
The $j$th mesh $\mathscr{M}(\mathbf{A},d-j; \mathbf{B}, j)$ has 
\[\min\left\{\Lambda\left(\frac{n}{2},d-j\right), \Lambda\left(\frac{n}{2},j\right) \right\}\]
layers. (In~\cite{Eisenbrand:Diameter-of-Polyhedra}, equation (4) should say $l = \min\{DC(m,i,i-1),DC(m,j,j-1)\}$.)

Theorem 4.1 in~\cite{Eisenbrand:Diameter-of-Polyhedra} provides a construction which shows that
\[\Lambda(q,r+1) \geq \left\lfloor \frac{q-r}{3 \ln q} \right\rfloor\]
Applying this theorem with $(q,r)=(\frac{n}{2},d-j)$ and $(q,r)=(\frac{n}{2},j)$ gives the desired result.
\end{proof}
Thus, the diameter $\delta(\Vn)=\ell_1 + \dots + \ell_{d-1}-1$ of $\Vn$ is in $\Omega(n^2/\log n)$, and therefore the Hirsch ratio of $\Vn$ is in $\Omega(n/\log n)$.

In~\cite{Eisenbrand:Diameter-of-Polyhedra} the layers $\V_{j,1}, \dots, \V_{j,\ell_j}$ are in the $j$th mesh $\mathscr{M}(\mathbf{A},d-j; \mathbf{B}, j)$. To obtain a sectioned layer family, let each layer be a vertex, and assign section $j$ to each layer of the form $\V_{j,k}$. In other words, the layers in the $j$th mesh of the connected layer family are the vertices in the $j$th section of the layer family.

Eisenbrand et al. note that if $C$ is a $c$-subset ($c \leq d-1$) of $\mathbf{A} \cup \mathbf{B}$ satisfying $|C \cap \mathbf{A}| \leq a$ and $|C \cap \mathbf{B}| \leq b$, then there is a $d$-set $E$ in each layer of the mesh $\mathscr(\mathbf{A},a;\mathbf{B},b)$ which contains $C$. By substituting $a=d-j$ and $b=j$, there is a $d$-set in each layer of the $j$th mesh containing $C$. That is, the family $\Vn$ satisfies sectioned-covering, which implies dimension reduction. 

Our goal is to modify this construction to preserve dimension reduction and add the properties of endpoint count and strong adjacency, while not reducing the diameter by more than a constant factor.  The steps to do so are analogous to those in Section~\ref{section:Eisenbrand1}, though more complicated because of the bigraded sets and the different conditions that apply to different sections. 

Before performing the modifications, we estimate the number of layers in the family $\Vn$.

\begin{lemma}\label{lemma:2nd-section-size}
Fix a real number $\epsilon > 0$. If $d$ is sufficiently large, $d=\frac{n}{4}$, and $\epsilon d < j < (1-\epsilon)d$, then the number of layers in section~$j$ of $\Vn$ is at least
\[\left\lfloor\frac{\frac{n}{4}}{3\ln\frac{n}{2}}\right\rfloor.\]
\end{lemma}

\begin{proof}
By Proposition~\ref{proposition:2nd-layers-in-mesh}, we find lower bounds for 
\[\left\lfloor \frac{\frac{n}{2}-(d-j-1)}{3 \ln \frac{n}{2}} \right\rfloor \; \textup{and} \; \left\lfloor \frac{\frac{n}{2}-(j-1)}{3 \ln \frac{n}{2}} \right\rfloor.\]

For the first expression, using the fact that $j > \epsilon d = \epsilon \cdot\frac{n}{4}$,
\[
\frac{\frac{n}{2}-(d-j-1)}{3 \ln \frac{n}{2}} 
=
\frac{\frac{n}{4}+j+1}{3 \ln \frac{n}{2}} 
>
\frac{ \frac{n}{4}+ \epsilon \cdot\frac{n}{4}+1}{3 \ln \frac{n}{2}}
>
\frac{\frac{n}{4}}{3\ln\frac{n}{2}}.
\]
Similarly, for the second expression, using $j < (1-\epsilon)d$ implies $-j>(\epsilon-1)\frac{n}{4}$, thus
\[
\frac{\frac{n}{2}-(j-1)}{3 \ln \frac{n}{2}} 
>
\frac{\frac{n}{2}+(\epsilon-1)\frac{n}{4}+1}{3 \ln \frac{n}{2}}
=
\frac{ \frac{n}{4}+ \epsilon \cdot\frac{n}{4}+1}{3 \ln \frac{n}{2}}
>
\frac{\frac{n}{4}}{3\ln\frac{n}{2}}.
\]
The desired inequality follows.
\end{proof}
This immediately implies:
\begin{corollary}\label{corollary:2nd-section-size-m}
Fix a natural number $m \geq 1$ and a real number $\epsilon > 0$. If $d$ is sufficiently large, $d=\frac{n}{4}$, and $\epsilon d < j < (1-\epsilon)d$, then section $j$ of $\Vn$ contains at least $m$ layers.\hfill\qedsymbol
\end{corollary}

Fix $m \geq 1$ and $\epsilon$ such that $0 < \epsilon < 1/2$. Only later we will need to fix $\epsilon$, and for concreteness, we will later pick 
\begin{equation}\label{equation:epsilon14}
\textstyle\epsilon:=\frac14.
\end{equation}

\subsection{Obtaining $m$-sectioned-covering}

Let $\overleftarrow{\jmath}=\lceil \epsilon d\rceil$ and $\overrightarrow{\jmath}=\lfloor (1-\epsilon)d\rfloor$. 
For each $j$ such that $\overleftarrow{\jmath} \leq j \leq \overrightarrow{\jmath}$, write $\ell_j = \delta_j m + r_j$ with $0 \leq r_j < m$. 
By Corollary~\ref{corollary:2nd-section-size-m}, $\delta_j \geq 1$ for each such $j$. Construct a new family 
\[\Wn = (\W_{\overleftarrow{\jmath},1}, \W_{\overleftarrow{\jmath},2}, \dots, \W_{\overleftarrow{\jmath},\delta_{\overleftarrow{\jmath}}}, \W_{\overleftarrow{\jmath}+1,1}, \dots, \W_{\overrightarrow{\jmath},1}, \W_{\overrightarrow{\jmath},2}, \dots, \W_{\overrightarrow{\jmath},\delta_{\overrightarrow{\jmath}}})\]
by merging consecutive groups of $m$ layers which are in the same section. For $\W_{j,\delta_j}$, merge the last $m+r_j$ layers.
Thus the diameter $\delta(\Wn)$ of the new layer family $\Wn$ satisfies:
\begin{proposition}\label{proposition:2nd-diameter-W}
The diameter $\delta(\Wn)$ of $\Wn$ is at least
\[\frac{1-2\epsilon}{48m} \cdot \frac{n^2}{\ln \frac{n}{2}}-1.\]
\end{proposition}
\begin{proof}
By Lemma~\ref{lemma:2nd-section-size}, each section $j$ of $\Vn$ between
$\overleftarrow{\jmath}=\lceil \epsilon d\rceil$ and $\overrightarrow{\jmath}=\lfloor (1-\epsilon)d\rfloor$ contains at least
\[\frac{\frac{n}{4}}{3\ln\frac{n}{2}}\]
layers, so each section of $\Wn$ contains at least
\[\frac1m \cdot \frac{\frac{n}{4}}{3\ln\frac{n}{2}}\]
layers. The total number of layers is at least
\[
\frac1m \sum_{\epsilon d < j < (1-\epsilon)d} \frac{\frac{n}{4}}{3 \ln \frac{n}{2}} 
= 
\frac1m\left((1-\epsilon)\frac{n}{4}-\epsilon\cdot\frac{n}{4}\right)\frac{\frac{n}{4}}{3\ln\frac{n}{2}}\]
thus the diameter is at least this minus one.
\end{proof}
For fixed $m$ and $\epsilon$, the diameter $\delta(\Wn)$ of $\Wn$ is in $\Omega(n^2/\log n)$, and therefore the Hirsch ratio of $\Wn$ is in $\Omega(n/\log n)$.

In the new layer family $\Wn$, the sectioning is determined by section number of the layers which were merged together, and so we still have well-defined sections $\overleftarrow{\jmath},\dots,\overrightarrow{\jmath}$. Then the family $\Wn$ satisfies $m$-sectioned covering (which implies dimension reduction).

If $m=1$, $\overleftarrow{\jmath}=1$, and $\overrightarrow{\jmath}=\delta(\Vn)$ then $\Wn=\Vn$. Only several steps later will we need to choose $m>1$ (in fact $m \geq 13$.)  From now, we will choose $\epsilon = \frac14$, so $\overleftarrow{\jmath} = \lceil{\frac{d}{4}}\rceil$ and $\overrightarrow{\jmath} = \lfloor{\frac{3d}{4}}\rfloor$. 

\begin{theorem}
The layer family $\Wn$ satisfies $m$-sectioned covering and thus dimension reduction. The diameter of $\Wn$ is in $\Omega(n^2/\log n)$.\hfill\qedsymbol
\end{theorem}

\subsection{Doubling}

For the next step, construct a new family
\[\Xn=(\X_{\overleftarrow{\jmath},1},\dots,\X_{\overleftarrow{\jmath},\delta_{\overleftarrow{\jmath}}},\X_{\overleftarrow{\jmath}+1,1},\dots,\X_{\overrightarrow{\jmath},\delta_{\overrightarrow{\jmath}}})\]
on $2n$ symbols of dimension $2d$ by duplicating a symbol every time it occurs in $\Wn$, exactly as in Section~\ref{section:Eisenbrand1}. That is,
\[\X_{j,k} = \{ D(C): C \in \W_{j,k}\} \]
and again we call each $2d$-set of the form $D(C) \subseteq [2] \times (\mathbf{A} \cup \mathbf{B})$ a \defn{correspondence set}.

\begin{theorem}
The layer family $\Xn$ has adjacency, endpoint count, $m$-sectioned-width-covering, and dimension reduction. The diameter of $\Xn$ is in $\Omega(n^2/\log n)$.\hfill\qedsymbol
\end{theorem}

\subsection{Obtaining linkage}

As in Section~\ref{section:Eisenbrand1}, our next step is to add interpolation sets to obtain linkage. The layer family $\Yn$ is obtained from $\Xn$ by adding one \defn{interpolation set} $I_{j,k}$ to each layer $\X_{j,k}$ to form $\Y_{j,k}$. We do not need to add an interpolation set to the very last layer $\X_{\overrightarrow{\jmath},\delta_{\overrightarrow{\jmath}}}$.

For each $j$ such that $\overleftarrow{\jmath} \leq j \leq \overrightarrow{\jmath}$ and for each $k$ such that $1 \leq k \leq \delta_j-1$, we define $I_{j,k}$ as follows. 
First choose a random $(d-j-1)$-subset $P_{j,k} \subseteq \mathbf{A}$ and a random $j$-subset $Q_{j,k} \subseteq \mathbf{B}$. 
By the sectioned-covering property for $\Wn$, there exist elements $f_{j,k},g_{j,k} \in \mathbf{A}$ such that
\[D(P_{j,k} \cup \{f_{j,k}\} \cup Q_{j,k}) \in \X_{j,k} \quad\text{and}\quad D(P_{j,k} \cup \{g_{j,k}\} \cup Q_{j,k}) \in \X_{j,k+1}.\]
As before, define 
\begin{equation}\label{equation:Fjk-Gjk}
F_{j,k} = P_{j,k} \cup \{f_{j,k}\} \quad\text{and}\quad G_{j,k} = P_{j,k} \cup \{g_{j,k}\}.
\end{equation}

Let $I_{j,k} \subseteq [2] \times (\mathbf{A} \cup \mathbf{B})$ be the $2d$-set whose first row is $F_{j,k} \cup Q_{j,k}$ and whose second row is $G_{j,k} \cup Q_{j,k}$. In other words, 
\begin{equation}\label{equation:Ijk}
I_{j,k} = \{(1,f_{j,k}), (2,g_{j,k})\} \cup D(P_{j,k}) \cup D(Q_{j,k}).
\end{equation}
Also set
\begin{equation}\label{equation:Ujk}
U_{j,k} = \pi(I_{j,k}) \cap \mathbf{A} = P_{j,k} \cup \{f_{j,k},g_{j,k}\}.
\end{equation}
Thus $U_{j,k}$ is a $(d-j+1)$-subset of $\mathbf{A}$.
See Figure~\ref{figure:mid-section-interpolation} for a schematic diagram.
\begin{figure}[hbt]
\begin{tikzpicture}
\draw [decorate,decoration={brace,amplitude=4pt},xshift=0pt,yshift=0pt]
(1.25,1.3) -- (3.15,1.3) node [black,midway,yshift=0.4cm]
{\footnotesize $d-j-1$};
\draw [decorate,decoration={brace,amplitude=4pt},xshift=0pt,yshift=0pt]
(3.25,1.3) -- (6.15,1.3) node [black,midway,yshift=0.4cm]
{\footnotesize $j$};
\draw [decorate,decoration={brace,amplitude=4pt},xshift=-4pt,yshift=0pt]
(-.1,0) -- (-.1,1.2) node [black,midway,xshift=-0.6cm] 
{$I_{j,k}$};
\draw (.6,.6) rectangle (1.2,1.2);
\node at (.9,.9) {$f_{j,k}$};
\draw (0,0) rectangle (.6,.6);
\node at (.3,.3) {$g_{j,k}$};
\draw (1.2,1.2) rectangle (3.2,.6);
\node at (2.2,.9) {$P_{j,k}$};
\draw (1.2,.6) rectangle (3.2,0);
\node at (2.2,.3) {$P_{j,k}$};
\draw (3.2,1.2) rectangle (6.2,.6);
\node at (4.7,.9) {$Q_{j,k}$};
\draw (3.2,.6) rectangle (6.2,0);
\node at (4.7,.3) {$Q_{j,k}$};
\node at (7,.9) {$\in \W_{j,k}$};
\node at (7,.3) {$\in \W_{j,k+1}$};
\draw (0,-1) rectangle (3.2,-.4);
\node at (1.6,-.7) {$U_{j,k}$};
\draw (3.2,-1) rectangle (6.2,-.4);
\node at (4.7,-.7) {$Q_{j,k}$};
\draw [decorate,decoration={brace,mirror,amplitude=4pt},xshift=0pt,yshift=0pt]
(0.1,-1.1) -- (3.15,-1.1) node [black,midway,yshift=-0.4cm]
{\footnotesize $\subseteq {\bf A}$};
\draw [decorate,decoration={brace,mirror,amplitude=4pt},xshift=0pt,yshift=0pt]
(3.25,-1.1) -- (6.15,-1.1) node [black,midway,yshift=-0.4cm]
{\footnotesize $\subseteq {\bf B}$};\end{tikzpicture}

\caption{Schematic diagram of an interpolation set $I_{j,k}$ and related sets when $k < \delta_j$}
\label{figure:mid-section-interpolation}
\end{figure}

The remaining interpolation sets $I_{j,\delta_j}$ are defined as follows. Again choose a random $(d-j-1)$-subset $P_{j,\delta_j} \subseteq \mathbf{A}$ and a random $j$-subset $Q_{j,\delta_j} \subseteq \mathbf{B}$. By the sectioned-covering property for $\Wn$, there exists an element $f_{j,\delta_j} \in \mathbf{A}$ such that 
\[D(P_{j,\delta_j} \cup \{f_{j,\delta_j}\} \cup Q_{j,\delta_j}) \in \X_{j,\delta_j}.\]
Also, there exists an element $h_{j,\delta_j} \in \mathbf{B}$ such that 
\[D(P_{j,\delta_j} \cup Q_{j,\delta_j} \cup \{h_{j,\delta_j}  \}) \in \X_{j+1,1}.\]
Now define the sets 
\begin{equation}\label{equation:Fjdeltaj-Gjdeltaj}
F_{j,\delta_j} = P_{j,\delta_j} \cup \{f_{j,\delta_j}\} \quad\text{and}\quad H_{j,\delta_j} = Q_{j,\delta_j} \cup \{h_{j,\delta_j}\}.
\end{equation}
So $F_{j,\delta_j}$ again is a $(d-j)$-subset of $\mathbf{A}$, but this time $H_{j,\delta_j}$ is a $(j+1)$-subset of $\mathbf{B}$. 

Let $I_{j,\delta_j} \subseteq [2] \times (\mathbf{A} \cup \mathbf{B})$ be the $2d$-set whose first row is the $(d-j,j)$-set $F_{j,\delta_j} \cup Q_{j,\delta_j}$ and whose second row is the $(d-j-1,j+1)$-set $P_{j,\delta_j} \cup H_{j,\delta_j}$. In other words, 
\begin{equation}\label{equation:Ijdeltaj}
I_{j,\delta_j} = \{(1,f_{j,\delta_j}), (2,h_{j,\delta_j})\} \cup D(P_{j,\delta_j}) \cup D(Q_{j,\delta_j}).
\end{equation}
Again set
\begin{equation}\label{equation:Ujdeltaj}
U_{j,\delta_j} = \pi(I_{j,\delta_j}) \cap \mathbf{A} = P_{j,\delta_j} \cup \{f_{j,\delta_j}\}.
\end{equation}
See Figure~\ref{figure:end-section-interpolation} for a schematic diagram.
Note that $U_{j,\delta_j}$ (unlike $U_{j,k}$ for $k < \delta_j$) is a $(d-j)$-subset of $\mathbf{A}$.
\begin{figure}[hbt]
\begin{tikzpicture}
\draw [decorate,decoration={brace,amplitude=4pt},xshift=0pt,yshift=0pt]
(1.25,1.3) -- (3.15,1.3) node [black,midway,yshift=0.4cm]
{\footnotesize $d-j-1$};
\draw [decorate,decoration={brace,amplitude=4pt},xshift=0pt,yshift=0pt]
(3.25,1.3) -- (6.15,1.3) node [black,midway,yshift=0.4cm]
{\footnotesize $j$};
\draw [decorate,decoration={brace,amplitude=4pt},xshift=-4pt,yshift=0pt]
(.5,0) -- (.5,1.2) node [black,midway,xshift=-0.6cm] 
{$I_{j,\delta_j}$};
\draw (.6,.6) rectangle (1.2,1.2);
\node at (.9,.9) {$f_{j,\delta_j}$};
\draw (6.2,0) rectangle (6.8,.6);
\node at (6.5,.3) {$h_{j,\delta_j}$};
\draw (1.2,1.2) rectangle (3.2,.6);
\node at (2.2,.9) {$P_{j,\delta_j}$};
\draw (1.2,.6) rectangle (3.2,0);
\node at (2.2,.3) {$P_{j,\delta_j}$};
\draw (3.2,1.2) rectangle (6.2,.6);
\node at (4.7,.9) {$Q_{j,\delta_j}$};
\draw (3.2,.6) rectangle (6.2,0);
\node at (4.7,.3) {$Q_{j,\delta_j}$};
\node at (7.6,.9) {$\in \W_{j,\delta_j}$};
\node at (7.6,.3) {$\in \W_{j+1,1}$};
\draw (.6,-1) rectangle (3.2,-.4);
\node at (1.9,-.7) {$U_{j,\delta_j}$};
\draw (3.2,-1) rectangle (6.8,-.4);
\node at (5.0,-.7) {$H_{j,\delta_j}$};
\draw [decorate,decoration={brace,mirror,amplitude=4pt},xshift=0pt,yshift=0pt]
(.6,-1.1) -- (3.15,-1.1) node [black,midway,yshift=-0.4cm]
{\footnotesize $\subseteq {\bf A}$};
\draw [decorate,decoration={brace,mirror,amplitude=4pt},xshift=0pt,yshift=0pt]
(3.25,-1.1) -- (6.8,-1.1) node [black,midway,yshift=-0.4cm]
{\footnotesize $\subseteq {\bf B}$};
\end{tikzpicture}
\caption{Schematic diagram of an interpolation set $I_{j,\delta_j}$ and related sets}
\label{figure:end-section-interpolation}
\end{figure}
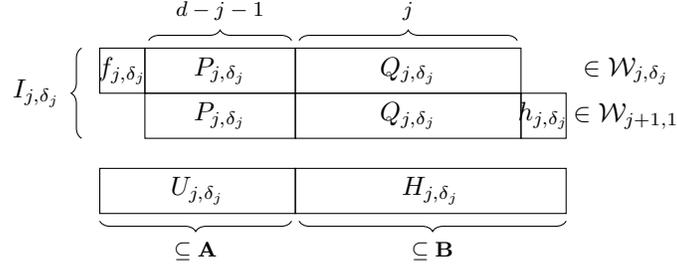

Define $\mathcal{U}$ to be the set of all the sets $U_{j,k}$. In other words,
\begin{equation}
\mathcal{U} = \{U_{j,k} \mid \overleftarrow{\jmath} \leq j \leq \overrightarrow{\jmath}, \, 1 \leq k < \delta_j \} \, \cup \, \{U_{j,\delta_j} : U_{j,k} \mid \overleftarrow{\jmath} \leq j < \overrightarrow{\jmath} \}
\end{equation}

Later we will refer to several conditions satisfied by the sets involved in constructing the interpolation sets, which we state now:
\begin{lemma}\label{lemma:2nd-interp-properties}
Let $\overleftarrow{\jmath} \leq j \leq \overrightarrow{\jmath}$ and $1 \leq k \leq \delta_j$.
\begin{enumerate}
\item\label{Pjk-to-Ujk} The set $U_{j,k}$ is obtained by adding either one or two elements of $\mathbf{A}$ to $P_{j,k}$.
\item\label{Pjk-sub-Ujk} For all $j$ and $k$, $P_{j,k} \subseteq U_{j,k} \subseteq \mathbf{A}$.
\item\label{Ujk-card} If $k < \delta_j$, then $U_{j,k}$ is a $(d-j+1)$-subset of $\mathbf{A}$. The set $U_{j,\delta_j}$ is a $(d-j)$-subset of $\mathbf{A}$.
\item\label{piIjk-is-Ujk} For all $(j,k)$, one has $\pi(I_{j,k}) \cap \mathbf{A} = U_{j,k}$.
\end{enumerate}
\end{lemma}
\begin{proof}$ $
\begin{enumerate}
\item If $k < \delta_j$, then from the definition of $U_{j,k}$ in~\eqref{equation:Ujk}, $U_{j,k}$ is obtained by adding two elements of $\mathbf{A}$ to $P_{j,k}$.
If $k = \delta_j$, then from the definition of $U_{j,\delta_j}$ in~\eqref{equation:Ujdeltaj}, $U_{j,\delta_j}$ is obtained by adding one element of $\mathbf{A}$ to $P_{j,k}$.
\item This follows from part~(\ref{Pjk-to-Ujk}).
\item This follows from the fact that $P_{j,k}$ is a $(d-j-1)$-subset of $\mathbf{A}$ and~\eqref{equation:Ujk} or~\eqref{equation:Ujdeltaj}.
\item This follows from the definition in equation~\eqref{equation:Ujk}.
\end{enumerate}
\end{proof}

%\begin{example}\label{example:18symbols}
%Suppose $n=18$ and $d=8$. Let $\mathbf{A}=\{a_1,a_2,\dots,a_9\}$ and $\mathbf{B}=\{b_1,b_2,\dots,b_9\}$. Consider $j=4$ and suppose that $\delta_4  = 6$. We will construct an interpolation set at the intermediate layer $(j,k)=(4,2)$ and an interpolation set at the last layer $(4,6)$ of section 4.

%Let $P_{4,2} = \{a_1,a_2,a_3\}$ and $Q_{4,2} = \{b_1,b_2,b_3,b_4\}$. Suppose we can choose $f_{2,4}=a_6$ and $g_{2,4} = a_8$. That is, we have 
%$$D(\{a_1,a_2,a_3,a_6,b_1,b_2,b_3,b_4\}) \in \X_{4,2} \quad\text{and}\quad D(\{a_1,a_2,a_3,a_8,b_1,b_2,b_3,b_4\}) \in \X_{4,2}.$$ 

%Then we define $\Y_{4,2} = \X_{4,2} \cup \{I_{4,2}\}$ where 
%$$I_{4,2} = \{(1,a_6), (2,a_8)\} \cup D(P_{4,2}) \cup D(Q_{4,2}).$$

%Let $P_{4,6} = \{a_2,a_5,a_7\}$ and $Q_{4,6} = \{b_3,b_5,b_6,b_7\}$. Suppose we can choose $f_{4,6}=a_9$ and $h_{4,6} = b_1$. That is, we have the doubled $(4,4)$-set 
%$$D(\{a_2,a_5,a_7,a_9,b_3,b_5,b_6,b_7\}) \in \X_{4,6}$$ 
%and the doubled $(3,5)$-set 
%$$D(\{a_2,a_5,a_7,b_1,b_3,b_5,b_6,b_7\}) \in \X_{5,1}.$$ 

%Then we define $\Y_{4,6} = \X_{4,6} \cup \{I_{4,6}\}$ where 
%$$I_{4,2} = \{(1,a_9), (2,b_1)\} \cup D(P_{4,6}) \cup D(Q_{4,6}).$$
%\end{example}

\begin{corollary} \label{corollary:Udiff}
For $(j,k) \neq (j',k')$, we have $\left| U_{j',k'} \setminus U_{j,k} \right| \geq 3$.
\end{corollary}

\begin{proof}
$\left| U_{j',k'} \setminus U_{j,k} \right|  \geq  \left| P_{j',k'} \setminus U_{j,k} \right| \geq \left| P_{j',k'} \setminus P_{j,k} \right| - 2  \geq 5 - 2 = 3$,
where the first inequality comes from Lemma~\ref{lemma:2nd-interp-properties}(\ref{Pjk-sub-Ujk}) applied to $(j',k')$, the second equality comes from Lemma~\ref{lemma:2nd-interp-properties}(\ref{Pjk-to-Ujk}), and the third is Lemma~\ref{lemma:Pjk-diff}. 
\end{proof}

\begin{corollary} \label{corollary:I-intersection}
For $(j,k) \neq (j',k')$, we have $\width(I_{j',k'} \cap I_{j,k}) \leq d-2$. 
\end{corollary}

\begin{proof}
We simply calculate
\begin{align*}
\width(I_{j',k'} \cap I_{j,k}) & =  \Awidth(I_{j',k'} \cap I_{j,k}) + \Bwidth(I_{j',k'} \cap I_{j,k}) \\
& \leq  \Awidth(I_{j',k'} \cap I_{j,k}) + \Bwidth(I_{j',k'}) \\
& =  \left| U_{j',k'} \cap U_{j,k} \right| + \Bwidth(I_{j',k'}) \\
& \leq  \left| U_{j,k} \right| - 3 + \Bwidth(I_{j',k'}) \\
& =  \Awidth(I_{j',k'}) - 3 + \Bwidth(I_{j',k'}) \\
& =  \width(I_{j',k'}) - 3 \\
& =  d-2.  
\end{align*}
where the second inequality follows from Corollary~\ref{corollary:Udiff}.
\end{proof}

%\begin{proposition}
%The layer family $\Yn$ has endpoint count.
%\end{proposition}
%\begin{proof}
%Let $S \subseteq [2] \times (\mathbf{A} \cup \mathbf{B})$ be a $(2d-1)$-subset. Thus $\width(S) \geq d$.

%If $\width(S) \geq d+2$, then no set in $\Zn$ contains $S$, since each correspondence set has width $d$ and each interpolation set has width $d+1$.

%If $\width(S) = d+1$, again no correspondence set is wide enough to contain $S$. If $I$ and $I'$ are distinct interpolation sets that both contain $S$, then $\pi(I) \cap \pi(I') \supseteq \pi(S)$, so 
%\[\width(I \cap I') \geq \width(S) = d+1,\]
%contradicting Corollary~\ref{corollary:I-intersection}, so at most one interpolation set in $\Zn$ contain $S$. 

%Suppose $\width(S) = d$. If $T$ is a correspondence set that contains $S$, then $\pi(T) \subseteq \pi(S)$. But both $\pi(T)$ and $\pi(S)$ have width $d$, so in fact $\pi(T) = \pi(S)$ and thus by the form of correspondence sets, $T = D(\pi(S))$. Thus no more than one correspondence set contains $S$. 
%If $I$ and $I'$ are distinct interpolation sets that both contain $S$, then by the same argument as in the previous statement, we see that 
%\[\width(I \cap I') \geq \width(S) = d,\]
%again contradicting Corollary~\ref{corollary:I-intersection}. So at most one interpolation set% contains $S$. 

%In all cases, $S$ is contained in at most one correspondence set and at most one interpolation% set in $\Yn$. Thus, $\Yn$ satisfies endpoint count.
%\end{proof}

\begin{proposition}\label{proposition:2ndlinkageinY}
The layer family $\Yn$ has linkage.
\end{proposition}
\begin{proof}
For each $j$ and for each $1 \leq k \leq \delta_{j-1}$, 
the sets $I_{j,k} \in \Y_{j,k}$ and $D(G_{j,k} \cup Q_{j,k}) \in \Y_{j,k+1}$ provide linkage between these layers in $\Yn$. 
For each $j$, the sets $I_{j,\delta_j} \in \Y_{j,\delta_j}$ and $D(P_{j,\delta_j} \cup H_{j,\delta_j}) \in \Y_{j+1,1}$ provide linkage between these layers in $\Yn$.
\end{proof}

\begin{theorem}
The layer family $\Yn$ has endpoint count and linkage. The diameter of $\Yn$ is in $\Omega(n^2/\log n)$.
\end{theorem}

\begin{proof} The only property we have not shown is endpoint count. Since we will later and independently prove endpoint count for the final family $\Zn$, we omit the proof here. 
\end{proof}

\subsection{Obtaining strong adjacency and dimension reduction}

As in Section~\ref{section:Eisenbrand1}, our final modification will be to remove certain correspondence sets in order to restore dimension reduction and adjacency, obtaining $\Zn$ from $\Yn$. Now we will need to remove two different types of correspondence sets. The first type is removed to restore dimension reduction within individual sections and is directly analogous to the sets $R_k^a$ in Section~\ref{section:Eisenbrand1}. The second type is removed to restore dimension reduction between different sections. 

First, for each pair $(j,k)$ with $k \leq \delta_{j}-1$ and for each $a \in P_{j,k}$, define  
\[R_{j,k}^a = D\left( \, (P_{j,k} \setminus \{a\}) \cup \{f_{j,k},g_{j,k}\} \cup Q_{j,k} \, \right).\]
See Figure~\ref{figure:Rjka}.
By completeness, $R_{j,k}^a \in \Y_{j,k'}$ for some $1 \leq k' \leq \delta_j$. 
\begin{figure}[hbt]
\begin{tikzpicture}
\draw [decorate,decoration={brace,amplitude=4pt},xshift=-4pt,yshift=0pt]
(-.1,0) -- (-.1,1.2) node [black,midway,xshift=-0.6cm] 
{$R_{j,k}^a$};
\draw (.6,.6) rectangle (1.2,1.2);
\node at (.9,.9) {$f_{j,k}$};
\draw (.6,0) rectangle (1.2,0.6);
\node at (.9,.3) {$f_{j,k}$};
\draw (0,0) rectangle (.6,.6);
\node at (.3,.3) {$g_{j,k}$};
\draw (0,.6) rectangle (.6,1.2);
\node at (.3,.9) {$g_{j,k}$};
\draw (1.2,1.2) rectangle (3.2,.6);
\node at (2,.9) {$P_{j,k}$};
\draw (1.2,.6) rectangle (3.2,0);
\node at (2,.3) {$P_{j,k}$};
\draw (3.2,1.2) rectangle (6.2,.6);

\draw[fill=lightgray, color=lightgray] (2.7,.7) rectangle (3.1,1.1);
\draw[fill=lightgray, color=lightgray] (2.7,.1) rectangle (3.1,0.5);
\node at (2.9,.3) {$a$};
\node at (2.9,.9) {$a$};

\node at (4.7,.9) {$Q_{j,k}$};
\draw (3.2,.6) rectangle (6.2,0);
\node at (4.7,.3) {$Q_{j,k}$};
\end{tikzpicture}
\caption{Schematic diagram of $R_{j,k}^a$, an $I_{j,k}$-resembler of type $a$ for $k < \delta_j$. The shaded element $a \in P_{j,k}$ is removed.}
\label{figure:Rjka}
\end{figure}

For each pair $(j,\delta_j)$, we remove \emph{two} families of correspondence sets that follow this pattern: one in section $j+1$ and one in section $j$. For each $a \in P_{j,\delta_j}$, define 
\[R_{j,\delta_j}^a = D\left( (P_{j,\delta_j} \setminus \{a\}) \cup \{f_{j,\delta_j},h_{j,\delta_j}\} \cup Q_{j,\delta_j} \right).\]
For each $b \in Q_{j,\delta_j}$, define 
\[R_{j,\delta_j}^b = D\left( P_{j,\delta_j} \cup \{f_{j,\delta_j},h_{j,\delta_j}\} \cup (Q_{j,\delta_j} \setminus \{b\}) \right).\]
See Figures~\ref{figure:Rjdeltaja} and~\ref{figure:Rjdeltajb}.
By completeness, $R_{j,\delta_j}^a \subseteq [2] \times (\mathbf{A} \cup \mathbf{B})$ is a correspondence set somewhere in section $j+1$ and $R_{j,\delta_j}^b$ is a correspondence set somewhere in section $j$. We call all of the sets $R_{j,k}^a$, $R_{j,\delta_j}^a$, and $R_{j,\delta_j}^b$ the \defn{$I_{j,k}$-resemblers} because they have $2d-1$ elements in common with $I_{j,k}$. We further call $R_{j,k}^a$ and $R_{j,\delta_j}^a$ the \defn{resemblers of type $a$} and call $R_{j,\delta_j}^b$ the \defn{resemblers of type $b$}. Note that the two correspondence sets that $I_{j,k}$ is designed to interpolate between are never $I_{j,k}$-resemblers.
\begin{figure}[hbt]
\begin{subfigure}[b]{0.45\textwidth}
\begin{tikzpicture}
\draw [decorate,decoration={brace,amplitude=4pt},xshift=-4pt,yshift=0pt]
(.5,0) -- (.5,1.2) node [black,midway,xshift=-0.6cm] 
{$R_{j,\delta_j}^a$};
\draw (.6,.6) rectangle (1.2,1.2);
\node at (.9,.9) {$f_{j,\delta_j}$};
\draw (.6,0) rectangle (1.2,.6);
\node at (.9,.3) {$f_{j,\delta_j}$};
\draw (6.2,0) rectangle (6.8,.6);
\node at (6.5,.3) {$h_{j,\delta_j}$};
\draw (6.2,.6) rectangle (6.8,1.2);
\node at (6.5,.9) {$h_{j,\delta_j}$};
\draw (1.2,1.2) rectangle (3.2,.6);
\node at (2.0,.9) {$P_{j,\delta_j}$};
\draw (1.2,.6) rectangle (3.2,0);
\node at (2.0,.3) {$P_{j,\delta_j}$};
\draw (3.2,1.2) rectangle (6.2,.6);
\node at (4.7,.9) {$Q_{j,\delta_j}$};
\draw (3.2,.6) rectangle (6.2,0);
\node at (4.7,.3) {$Q_{j,\delta_j}$};
\draw[fill=lightgray, color=lightgray] (2.7,.7) rectangle (3.1,1.1);
\draw[fill=lightgray, color=lightgray] (2.7,.1) rectangle (3.1,0.5);
\node at (2.9,.3) {$a$};
\node at (2.9,.9) {$a$};
\end{tikzpicture}
\caption{Diagram of $R_{j,\delta_j}^a$, an $I_{j,\delta_{j}}$-resembler of type $a$}
\label{figure:Rjdeltaja}
\end{subfigure}
\begin{subfigure}[b]{0.45\textwidth}
\begin{tikzpicture}
\draw [decorate,decoration={brace,amplitude=4pt},xshift=-4pt,yshift=0pt]
(.5,0) -- (.5,1.2) node [black,midway,xshift=-0.6cm] 
{$R_{j,\delta_j}^b$};
\draw (.6,.6) rectangle (1.2,1.2);
\node at (.9,.9) {$f_{j,\delta_j}$};
\draw (.6,0) rectangle (1.2,.6);
\node at (.9,.3) {$f_{j,\delta_j}$};
\draw (6.2,0) rectangle (6.8,.6);
\node at (6.5,.3) {$h_{j,\delta_j}$};
\draw (6.2,.6) rectangle (6.8,1.2);
\node at (6.5,.9) {$h_{j,\delta_j}$};
\draw (1.2,1.2) rectangle (3.2,.6);
\node at (2.2,.9) {$P_{j,\delta_j}$};
\draw (1.2,.6) rectangle (3.2,0);
\node at (2.2,.3) {$P_{j,\delta_j}$};
\draw (3.2,1.2) rectangle (6.2,.6);
\node at (4.7,.9) {$Q_{j,\delta_j}$};
\draw (3.2,.6) rectangle (6.2,0);
\node at (4.7,.3) {$Q_{j,\delta_j}$};
\draw[fill=lightgray, color=lightgray] (3.5,.7) rectangle (3.9,1.1);
\draw[fill=lightgray, color=lightgray] (3.5,.1) rectangle (3.9,0.5);
\node at (3.7,.3) {$b$};
\node at (3.7,.9) {$b$};
\end{tikzpicture}
\caption{Diagram of $R_{j,\delta_j}^b$, an $I_{j,\delta_{j}}$-resembler of type $b$}
\label{figure:Rjdeltajb}
\end{subfigure}
\caption{Schematic diagrams of $I_{j,\delta_{j}}$-resemblers. Shaded elements are removed.}
\label{figure:Rjdeltaj}
\end{figure}

Define
\begin{equation}
\begin{array}{lll}
\mathcal{R} &=
&\{R_{j,k}^a \mid \overleftarrow{\jmath} \leq j \leq \overrightarrow{\jmath}, 1 \leq k \leq \delta_j-1, a \in P_{j,k}\}\\
&\cup
&\{R_{j,\delta_j}^a \mid \overleftarrow{\jmath} \leq j \leq \overrightarrow{\jmath}, a \in P_{j,\delta_j} \}\\
&\cup
&\{R_{j,\delta_j}^b \mid \overleftarrow{\jmath} \leq j \leq \overrightarrow{\jmath}, b \in Q_{j,\delta_j} \}.
\end{array}
\end{equation}

The second type of set we will remove is as follows. For each interpolation set $I_{j,k}$ with $1 \leq k \leq \delta_j-1$, define an \defn{$I_{j,k}$-enveloper} to be any correspondence set $D(C)$ such that $C \supseteq U_{j,k}$. From Figure~\ref{figure:mid-section-interpolation}, note that if $D(C)$ is an $I_{j,k}$-enveloper, then $\Awidth(D(C)) \geq d-j+1$; that is, $D(C)$ can only appear in $\Y_{j',k'}$ if $j' < j$. We do not define $I_{j,\delta_j}$-envelopers.

\begin{proposition}\label{proposition:2nd-removed-set-properties}$ $
\begin{enumerate}
\item\label{resembler-P} If $D(C)$ is an $I_{j,k}$-resembler of type $a$, then $C$ contains every element of $P_{j,k}$ but one. If $D(C)$ is an $I_{j,k}$-resembler of type $b$, then $C$ contains every element of $P_{j,k}$.
\end{enumerate}\hfill\qedsymbol
\end{proposition}

We are now ready to define the final layer family $\Zn$ that will satisfy all of the desired properties. We define $\Zn$ by removing from $\Yn$ all $I_{j,k}$-resemblers and all $I_{j,k}$-envelopers for all values of $j$ and $k$.  

%\begin{example}
%Continuing Example~\ref{example:18symbols} above, the $I_{4,2}$-resemblers are the following correspondence sets in section 4 of $\Yn$:
%$$R_{4,2}^1 = D(\{a_2,a_3,a_6,a_8,b_1,b_2,b_3,b_4\})$$
%$$R_{4,2}^2 = D(\{a_1,a_3,a_6,a_8,b_1,b_2,b_3,b_4\})$$
%$$R_{4,2}^3 = D(\{a_1,a_2,a_6,a_8,b_1,b_2,b_3,b_4\}).$$ 

%The $I_{2,4}$-envelopers in sections 1, 2, and 3 of $\Yn$ are the sets $D(C)$ where $C = \{a_1,a_2,a_3,a_6,a_8\} \cup C'$ and $C'$ is any $3$-subset of $\mathbf{A} \cup \mathbf{B}$. For example, to construct $\Zn$, the correspondence set 
%$D(a_1,a_2,a_3,a_6,a_8,b_1,b_5,b_9)$ must be removed from section 3 and the correspondence set $D(a_1,a_2,a_3,a_5,a_6,a_8,b_2,b_5)$ must be removed from section 2. 

%The $I_{4,6}$-resemblers in section 4 of $\Yn$ are the following:
%$$D(\{a_2,a_5,a_7,a_9,b_1,b_5,b_6,b_7\}),$$
%$$D(\{a_2,a_5,a_7,a_9,b_1,b_3,b_6,b_7\}),$$
%$$D(\{a_2,a_5,a_7,a_9,b_1,b_3,b_5,b_7\}),$$
%$$D(\{a_2,a_5,a_7,a_9,b_1,b_3,b_5,b_6\}).$$

%The resemblers in section 5 of $\Yn$ are the following: 
%$$D(\{a_5,a_7,a_9,b_1,b_3, b_5,b_6,b_7\}),$$ 
%$$D(\{a_2,a_7,a_9,b_1,b_3, b_5,b_6,b_7\}),$$
%$$D(\{a_2,a_5,a_9,b_1,b_3, b_5,b_6,b_7\}).$$

%By definition, there are no $I_{4,6}$-envelopers.
%\end{example}

\begin{lemma} \label{lemma:avoidUs}
Let $C$ be a subset of $\mathbf{A}$ with $\left| C \right| \leq \frac{3d}{4}$. Then there are at most three sets of the form $U_{j,k}$ in $\mathcal{U}$ such that $\left| U_{j,k} \setminus C \right| \leq 1$.  
\end{lemma}

\begin{proof}
Suppose there are four such sets $U_1, U_2, U_3, U_4$. Since each $U_i$ contains at most one element outside $C$, their union contains at most four elements outside $C$. For each $i=1,\dots,4$, the set $U_i$ is obtained by adding one or two elements to $P_i$ by Lemma~\ref{lemma:2nd-interp-properties}(\ref{Pjk-to-Ujk}), where $P_i$ is one of the sets $P_{j,k}$. Then by Lemma~\ref{lemma:Pjk-union}, 
\begin{align*}
\frac{13d}{16} & \leq  \left| P_1 \cup P_2 \cup P_3 \cup P_4 \right| \\ 
& \leq  \left| U_1 \cup U_2 \cup U_3 \cup U_4 \right|   \\
& \leq  \left| C \right| + 4 \\
& \leq  \frac{3d}{4} + 4,
\end{align*} 
which is a contradiction for $d$ large. 
\end{proof}

\begin{lemma}\label{lemma:2ndwidthclassificationinZ}
Let $S \subseteq [2] \times (\mathbf{A} \cup \mathbf{B})$. 
\begin{enumerate}
\item If $\width(S) \geq d+2$, then no set in $\Zn$ contains $S$. 
\item If $\width(S) = d+1$, then no correspondence set and at most one interpolation set in $\Zn$ contain $S$. 
\item If $\width(S) = d$, then at most one correspondence set and at most one interpolation set in $\Zn$ contain~$S$. 
\end{enumerate}
Thus, $S$ is contained in at most one correspondence set and at most one interpolation set in $\Zn$.
\end{lemma}

\begin{proof} $ $
\begin{enumerate}
\item Suppose $\width(S) \geq d+2$. Each correspondence set has width $d$ and each interpolation set has width $d+1$, so none contain $S$. 
\item Suppose $\width(S) = d+1$. Again no correspondence set is wide enough to contain $S$. If $I$ and $I'$ are distinct interpolation sets that both contain $S$, then $\pi(I) \cap \pi(I') \supseteq \pi(S)$, so 
\[\width(I \cap I') \geq \width(S) = d+1,\]
contradicting Corollary~\ref{corollary:I-intersection}.
\item Suppose $\width(S) = d$. If $T$ is a correspondence set that contains $S$, then $\pi(T) \supseteq \pi(S)$. But both $\pi(T)$ and $\pi(S)$ have width $d$, so in fact $\pi(T) = \pi(S)$ and thus by the form of correspondence sets, $T = D(\pi(S))$. Thus no more than one correspondence set contains $S$. 

If $I$ and $I'$ are distinct interpolation sets that both contain $S$, then by the same argument as in the previous statement, we see that 
\[\width(I \cap I') \geq \width(S) = d,\]
again contradicting Corollary~\ref{corollary:I-intersection}. So at most one interpolation set contains $S$.  
\end{enumerate}
Thus, $S$ is contained in at most one correspondence set and at most one interpolation set in $\Zn$.
\end{proof}

\begin{proposition} \label{proposition:2ndendpointcountinZ}
The layer family $\Zn$ has endpoint count. 
\end{proposition}

\begin{proof}
Let $S \subseteq [2] \times ({\bf A} \cup {\bf B})$ be a set of size $2d-1$. Then one row of $S$ has cardinality at least $d$, so $\width(S) \geq d$. By considering each of the three cases of Lemma~\ref{lemma:2ndwidthclassificationinZ}, we see that at most two sets in $\Zn$ contain $S$, verifying endpoint count. 
\end{proof}

\begin{lemma} \label{lemma:2ndadjacentlayerZ}
Let $S$ be a subset of $[2] \times (\mathbf{A} \cup \mathbf{B})$ of width $d$. Then either $S$ is contained in at most one $2d$-set of $\Zn$, or it is contained in exactly two $2d$-sets of $\Zn$ that occur in the same or adjacent layers.  
\end{lemma}

\begin{proof}
By Lemma~\ref{lemma:2ndwidthclassificationinZ}, the subset $S$ is contained in at most one correspondence set $D(C)$ and at most one interpolation set $I$ of $\Zn$. If $S$ is not contained in any interpolation set, then we are done, since $S$ is contained in at most one correspondence set. So assume some interpolation set $I$ contains $S$. Every interpolation set has width $d+1$, so $\pi(I) = \pi(S) \cup \{c\}$ for some single element $c \in \mathbf{A} \cup \mathbf{B}$. Also, every correspondence set has width $d$, so if $S$ is contained in a correspondence set, then that set is $D(\pi(S))$. 

The interpolation set $I$ is either $I_{j,k} \in \Z_{j,k}$ for some $j$ and some $k \leq \delta_j-1$, or it is $I_{j,\delta_j}$ for some $j$.
\begin{itemize}
\item Case 1: Suppose $I = I_{j,k} \in \Z_{j,k}$ for some $j$ and some $k \leq \delta_j-1$. Then $\pi(I_{j,k}) = P_{j,k} \cup \{f_{j,k},g_{j,k}\} \cup Q_{j,k}$. Recall that $Q_{j,k} \subseteq \mathbf{B}$ and the rest of the elements of $\pi(I_{j,k})$ belong to $\mathbf{A}$.  

\begin{itemize}
\item
Case 1a: Suppose $c = a \in P_{j,k}$. 

Then $\pi(S) = \pi(R_{j,k}^a)$. Thus $D(\pi(S)) = R_{j,k}^a$ which is a set that is removed from $\Zn$, so no correspondence set in $\Zn$ contains~$S$.

\item
Case 1b: Suppose $c = b \in Q_{j,k}$. 

Then $\pi(S) \supseteq U_{j,k}$. So by the removal of $I_{j,k}$-envelopers, no correspondence set in $\Zn$ contains~$S$.

\item
Case 1c: Suppose $c = f_{j,k}$. 

Then $\pi(S) = G_{j,k} \cup Q_{j,k}$, so the correspondence set $D(\pi(S)) \in \Z_{j,k+1}$. 

\item
Case 1d: Suppose $c = g_{j,k}$. 

Then $\pi(S) = F_{j,k} \cup Q_{j,k}$, so the correspondence set $D(\pi(S)) \in \Z_{j,k}$.
\end{itemize}

\item
Case 2: Suppose $I = I_{j,\delta_j} \in \Z_{j,\delta_j}$ for some $j$. Then $\pi(I_{j,\delta_j}) = P_{j,\delta_j} \cup \{f_{j,\delta_j}\} \cup Q_{j,\delta_j} \cup \{h_{j,\delta_j}\}$. Recall that $P_{j,\delta_j} \cup \{f_{j,\delta_j}\} \subseteq \mathbf{A}$ and $Q_{j,\delta_j} \cup \{h_{j,\delta_j}\} \subseteq \mathbf{B}$.   

\begin{itemize}
\item
Case 2a: Suppose $c = a \in P_{j,\delta_j}$. 

Then just as in case 1a, we have $\pi(S) = \pi(R_{j,\delta_j}^a)$, so no correspondence set in $\Zn$ contains $S$. 

\item
Case 2b: Suppose $c = b \in  Q_{j,\delta_j}$.

Then $\pi(S) = \pi(R_{j,\delta_j}^b)$ which is again removed in the construction of $\Zn$. So no correspondence set in $\Zn$ contains $S$. 

\item
Case 2c: Suppose $c = f_{j,\delta_j}$.

Then $\pi(S) = P_{j,\delta_j} \cup Q_{j,\delta_j} \cup \{h_{j,\delta_j}\}$, so $D(\pi(S)) \in \Z_{j+1,1}$.

\item
Case 2d: Suppose $c = h_{j,\delta_j}$.

Then $\pi(S) = P_{j,\delta_j} \cup \{f_{j,\delta_j}\} \cup Q_{j,\delta_j}$, so $D(\pi(S)) \in \Z_{j,\delta_j}$. 
\end{itemize}

\end{itemize}
In all cases, if $S$ is contained in both an interpolation set and a correspondence set, then the correspondence set is in the same or adjacent layer.
\end{proof}

\begin{proposition} \label{proposition:2ndadjacencyinZ}
The layer family $\Zn$ has adjacency.
\end{proposition}

\begin{proof}
This follows in the same manner as in Section~\ref{section:Eisenbrand1}. Let $S$ be a $(2d-1)$-subset of $[2] \times ({\bf A} \cup {\bf B})$. Then $\width(S) \geq d$. If $\width(S) \geq d+1$, then the first or the second cases of Lemma~\ref{lemma:2ndwidthclassificationinZ} imply that $S$ is contained in at most one set in $\Zn$. Adjacency immediately follows. Finally, if $\width(S) = d$, then Lemma~\ref{lemma:2ndadjacentlayerZ} implies that $S$ is active on only one layer or on two adjacent layers of $\Zn$, so again adjacency follows.
\end{proof}

\begin{proposition} \label{proposition:2ndlinkageinZ}
The layer family $\Zn$ has strong adjacency. 
\end{proposition}

\begin{proof}
To show strong adjacency in $\Zn$, it suffices to show that none of these sets which were created linkage in $\Yn$ in the proof of Proposition~\ref{proposition:2ndlinkageinY} are removed in the construction of $\Zn$. By definition, no interpolation set is removed, so we only need to consider the correspondence sets $D(G_{j,k} \cup Q_{j,k}) \in \Y_{j,k+1}$ and $D(P_{j,\delta_j} \cup H_{j,\delta_j}) \in \Y_{j+1,1}$ and show that these remain in the respective layers $\Z_{j,k+1}$ and $\Z_{j+1,1}$.

Let $L$ be one of these linking correspondence sets. 
Then $\pi(L) \cap \mathbf{A}$ is either $G_{j,k} \cup \{f_{j,k}\}$ in the first case or $P_{j,\delta_j}$ in the second case.
Thus $\pi(L)$ contains $P_{j,k}$ and either zero or one additional elements of $\mathbf{A}$.
Suppose, for a contradiction, that for $(j',k') \not= (j,k)$ the set $L$ is an $I_{j',k'}$-enveloper. Then $U_{j',k'} \subseteq \pi(L)$, and thus $U_{j',k'} \subseteq \pi(L) \cap \mathbf{A}$.
\begin{itemize}
\item Case 1: Suppose $L = D(G_{j,k} \cup Q_{j,k})$.

Then $U_{j',k'} \subseteq \pi(L) \cap \mathbf{A} = P_{j,k} \cup \{f_{j,k}, g_{j,k}\}$ and by Lemma~\ref{lemma:2nd-interp-properties}(\ref{Pjk-to-Ujk}), $U_{j',k'}$ is $P_{j',k'} \cup \{f_{j',k'}, g_{j',k'}\}$ since envelopers are not defined for $I_{j,\delta_j}$, and thus
\[ P_{j',k'} \cup \{f_{j',k'}, g_{j',k'}\} \subseteq P_{j,k} \cup \{f_{j,k}, g_{j,k}\},\]
which implies that
\[\left| P_{j,k} \setminus P_{j',k'} \right| \leq 4,\]
contradicting Lemma~\ref{lemma:Pjk-diff}. 

\item Case 2: Suppose $L = D(P_{j,\delta_j} \cup H_{j,\delta_j})$.

Then $U_{j',k'} \subseteq \pi(L) \cap \mathbf{A} = P_{j,\delta_j} \cup \{f_{j,\delta_j}\}$ 
and by Lemma~\ref{lemma:2nd-interp-properties}(\ref{Pjk-to-Ujk}), $U_{j',k'}$ is $P_{j',k'} \cup \{f_{j',k'}, g_{j',k'}\}$ since envelopers are not defined for $I_{j,\delta_j}$, and thus
\[ P_{j',k'} \cup \{f_{j',k'}, g_{j',k'}\} \subseteq P_{j,\delta_j} \cup \{f_{j,\delta_j}\},\]
which implies that
\[\left| P_{j,k} \setminus P_{j',k'} \right| \leq 3,\]
contradicting Lemma~\ref{lemma:Pjk-diff}. 
\end{itemize}

We now show that $L$ cannot be a $I_{j',k'}$-resembler for $(j',k')\not= (j,k)$. If it were, then we would have 
$$\left| L \cap I_{j,k} \right| = 2d-1 = \left| L \cap I_{j',k'} \right|$$  
which would imply $\left| I_{j,k} \cap I_{j',k'} \right| \geq 2d-2$ and so $\width(I_{j,k} \cap I_{j',k'}) \geq d-1$, contradicting Corollary~\ref{corollary:I-intersection}. 
\end{proof}

It only remains to show dimension reduction for the family $\Zn$ in Proposition~\ref{prop:2nddimensionreduction}. Since the family $\Zn$ does not satisfy any simple variant of the covering property, we first need to prove several technical statements which will replace covering in the final proof.

\begin{lemma} \label{lemma:extendtodminus1}
Let $C$ be a subset of  $\mathbf{A} \cup \mathbf{B}$ such that $\left| C \right| \leq d-1$ and such that $C$ does not contain any of the sets $U_{j, k} \in \mathcal{U}$.
Fix $j' \in \{\overleftarrow{\jmath},\dots,\overrightarrow{\jmath}-1\}$, where
$\overleftarrow{\jmath}=\lceil \epsilon d\rceil$, $\overrightarrow{\jmath}=\lfloor (1-\epsilon)d\rfloor$, and $\epsilon=\frac14$.
If $\left| C^A \right| \leq d-j'$ and $\left| C^B \right| \leq j'$,
then there exists a $(d-1)$-set $\widetilde{C} \subseteq \mathbf{A} \cup \mathbf{B}$ such that:
\begin{enumerate} 
\item $C \subseteq \widetilde{C}$, 
\item $\left| \widetilde{C}^A \right| \leq d-j'$ and $\left| \widetilde{C}^B \right| \leq j'$, and 
\item $\widetilde{C}$ does not contain any of the sets $U_{j, k} \in \mathcal{U}$.
\end{enumerate}
\end{lemma}

\begin{proof}
Either $\left| C^A \right|$ is equal to $d-j'$ or strictly less than it.
\begin{itemize}
\item
Case 1: Suppose $\left| C^A \right| = d-j'$.

Then we must have $\left| C^B \right| \leq j'-1$. Add arbitrary elements of $\mathbf{B}$ to create a $(d-j',j'-1)$-set $\widetilde{C}$. By Lemma~\ref{lemma:2nd-interp-properties}(\ref{Pjk-sub-Ujk}), each $U_{j,k} \in \mathcal{U}$ consists entirely of elements of ${\bf A}$, thus the last condition will never be violated.  

\item
Case 2: Suppose $\left| C^A \right| < d-j'$.

Due to~\eqref{equation:epsilon14}, $\epsilon=\frac14$ thus $\left| C^A \right| < \frac{3d}{4}$. By hypothesis, $C^A$ does not contain any $U_{j,k} \in \mathcal{U}$. On the other hand, by Lemma~\ref{lemma:avoidUs}, there are most three sets $U_{j,k} \in \mathcal{U}$ such that $\left| U_{j,k} \setminus C^A \right| = 1$. Thus there are at most three elements of $\mathbf{A}$ which, when added to $C^A$, would cause it to contain a set $U_{j,k}$. If $\left| C^A \right| < d-j'-1$, add any other element of $\mathbf{A}$. Repeat until we have exactly $d-j'-1$ elements of $\mathbf{A}$, which we can do due to~\eqref{equation:epsilon14}. Then (as in Case 1) add arbitrary elements of $\mathbf{B}$ until we have exactly $j'$ elements of $B$. Let $\widetilde{C}$ be the resulting $(d-j'-1,j')$-set.
\end{itemize}
\end{proof}

\begin{proposition}\label{proposition:2ndcovering}
Let $C$ be a subset of $\mathbf{A} \cup \mathbf{B}$ such that $\left| C \right| \leq d-1$ and such that $C$ does not contain any of the sets $U_{j, k} \in \mathcal{U}$. Let $\overleftarrow{\jmath} \leq j' \leq \overrightarrow{\jmath}$ and $1 \leq k' \leq \delta_{j'}$.  
If $\left| C^A \right| \leq d-j'$ and $\left| C^B \right| \leq j'$, then there is a correspondence set in $\Z_{j',k'}$ that contains $D(C)$.
\end{proposition}
\begin{proof}
Let $\widetilde{C}$ be the $(d-1)$-set obtained from $C$ using Lemma~\ref{lemma:extendtodminus1}. Then it will be sufficient to show that there is a correspondence set in $\Z_{j',k'}$ that contains $D(\widetilde{C})$. 

Now by the $m$-sectioned covering property for the layer family $\Wn$, there are $m$ distinct $d$-sets $E_1, \dots, E_m \in \W_{j',k'}$ that each contain $\widetilde{C}$. So the doubled sets $D(E_1), \dots, D(E_m) \in \Y_{j',k'}$ each contain $D(\widetilde{C})$. It will suffice to show that not all $m$ of these sets are removed in the construction of the layer family $\Zn$. The set $\widetilde{C}$ is either a $(d-j',j'-1)$-set or a $(d-j'-1,j')$-set.
\begin{itemize}
\item
Case 1: Suppose  $\widetilde{C}$ is a $(d-j', j'-1)$-set.

Since each $E_i$ is a $(d-j',j')$-set that contains $\widetilde{C}$, we have 
\begin{equation}\label{equation:EiACA}
E_i^A = \widetilde{C}^A \quad\text{and}\quad E_i^B = \widetilde{C}^B \cup \{b_i\}
\end{equation}
for some elements $b_1, \dots, b_m$ of $\mathbf{B}$. 

By Lemma~\ref{lemma:extendtodminus1}, the sets $E_i$ do not contain any $U_{j,k} \in \mathcal{U}$. In particular, $D(E_i)$ is not an $I_{j,k}$-enveloper for any values of $j$ and $k$. 

\textbf{Claim:} There are at most three pairs $(j,k)$ such that some $D(E_i)$ is an $I_{j,k}$-resembler.  

\textbf{Proof:} Suppose there are four such pairs $(j_1,k_1), \dots, (j_4,k_4)$. If $D(E_i)$ is an $I_{j,k}$-resembler, then by Proposition~\ref{proposition:2nd-removed-set-properties}(\ref{resembler-P}), each $E_i$ contains every element of $P_{j,k}$ (if type $b$) or every element of $P_{j,k}$ but one (if type $a$). We now calculate      
\[
\frac{3d}{4}
= \left| C^A \right|
= \left| \bigcup_{i=1}^p E_i^A \right|
\geq \left| \bigcup_{i=1}^4 P_{j_i,k_i} \right| - 4 
\geq \frac{13d}{16} - 4,
\]
where 
the first equality comes from~\eqref{equation:epsilon14},
the second equality comes from~\eqref{equation:EiACA}, and 
the last inequality comes from Lemma~\ref{lemma:Pjk-union}. This gives a contradiction when $d$ is large. 
Thus the claim is proved.

Now fix a pair $(j,k)$. Note that all $I_{j,k}$-resemblers of type $a$ have the same $\mathbf{B}$-part. But due to~\eqref{equation:EiACA} all of the sets $E_i$ have the same $\mathbf{A}$-part. We conclude that only one of the sets $D(E_i)$ can be an $I_{j,k}$-resembler of type $a$. 

Now suppose $k = \delta_j$ and consider the possibility that one or more of the $D(E_i)$ is an $I_{j,\delta_j}$-resembler of type $b$. If so, then $j = j'+1$ (to get the right-size $\mathbf{B}$-part, since $H_{j,\delta_j}$ is a $(j+1)$-subset of $\mathbf{B}$, thus $H_{j,\delta_j} \setminus \{b\}$ has cardinality $j$). For each such $E_i$, we have $\widetilde{C}^B \subseteq E_i^B \subseteq Q_{j,\delta_j}$, where $\left|\widetilde{C}^B\right| = j'-1$, $\left| E_i^B \right| = j'$, and $\left| Q_{j,\delta_j} \right| = j = j'+1$. 
There are exactly two $j$-sets $E \subseteq \mathbf{B}$ which satisfy $\widetilde{C}^B \subseteq E \subseteq Q_{j,\delta_j}$.
So at most two of the sets $E_i$ satisfy $\widetilde{C}^B \subseteq E_i^B \subseteq Q_{j,\delta_j}$.

From the claim, there are at most three pairs $(j,k)$ such that some $D(E_i)$ is an $I_{j,k}$ resembler.
For a fixed pair $(j,k)$, there are at most three (at most one of type $a$ and at most two of type $b$) sets of the form $D(E_i)$ that can be an $I_{j,k}$-resembler. Thus, there are at most nine sets of the form $D(E_i)$ which are resemblers.

In conclusion, for $m \geq 10$, we know that at least one $E_i$ is neither an enveloper nor a resembler, and hence belongs to $\Z_{j',k'}$.    

\item
Case 2: Suppose $\widetilde{C}$ is a $(d-j'-1,j')$-set.

Each $E_i$ is a $(d-j',j')$-set which contains the $(d-j'-1,j')$-set $\widetilde{C}$, thus
\begin{equation}\label{equation:EiACAa}
E_i^A = \widetilde{C}^A \cup \{a_i\} \quad\text{and}\quad E_i^B = \widetilde{C}^B 
\end{equation}
for some elements $a_1, \dots, a_m$ of $\mathbf{A}$. 

Suppose $D(E_i)$ is an $I_{j,k}$-enveloper. Then $E_i \supseteq U_{j,k}$. By Lemma~\ref{lemma:extendtodminus1}, $\widetilde{C}$ does not contain $U_{j,k} \subseteq \mathbf{A}$, so $E_i = C \cup \{a\}$ where $U_{j,k} \setminus C = \{a\}$. In particular, for each pair $(j,k)$, there can be at most one $i$ such that $E_i$ is an $I_{j,k}$-enveloper. Furthermore, by Lemma~\ref{lemma:avoidUs}, there are at most three sets of the form $U_{j,k} \in \mathcal{U}$ that satisfy $\left| U_{j,k} \setminus C \right| \leq 1$. So at most three of the sets $E_i$ are envelopers. 

\textbf{Claim:} There are at most three pairs $(j,k)$ such that some $E_i$ is an $I_{j,k}$-resembler.  

%fix start
\textbf{Proof:}
Suppose there are four such pairs $(j_1,k_1), \dots, (j_4,k_4)$. If $D(E_i)$ is an $I_{j,k}$-resembler, then by Proposition~\ref{proposition:2nd-removed-set-properties}(\ref{resembler-P}), each $E_i$ contains every element of $P_{j,k}$ (if type $b$) or every element of $P_{j,k}$ but one (if type $a$). We now calculate      
\[
\frac{3d}{4} = \left| C^A \right| =  \left| \bigcup_{i=1}^p E_i^A \right| - 4 \geq  \left| \bigcup_{i=1}^4 P_{j_i,k_i} \right| - 8  \geq  \frac{13d}{16} - 8,
\]
where 
the first equality comes from~\eqref{equation:epsilon14},
the second equality comes from~\eqref{equation:EiACAa}, and 
the last inequality comes from Lemma~\ref{lemma:Pjk-union}. This gives a contradiction when $d$ is large. 
Thus the claim is proved.
% fix end

Now fix a pair $(j,k)$. If $k=\delta_j$, so that there exist $I_{j,k}$-resemblers of type $b$, then note that all $I_{j,k}$-resemblers of type $b$ have the same $\mathbf{A}$-part. But all of the sets $E_i$ have the same $\mathbf{B}$-part. We conclude that only one can be an $I_{j,k}$-resembler of type $b$. 

Suppose one or more of the $E_i$ is an $I_{j,k}$-resembler of type $a$. If so, then $d-j = d-j'$ (that is, $j=j'$) in order to get the right-size $\mathbf{A}$-part). For each such $E_i$, we have $\widetilde{C}^A \subseteq E_i^A \subseteq P_{j,k}$, where $\left|\widetilde{C}^A \right| = d-j'-1$, $\left| E_i^B \right| = d-j'$, and $\left| P_{j,k} \right| = d-j'+1$. So there can be only two such $E_i$. 

Just as in case 1, we conclude that at most nine (at most three resemblers for each pair $(j,k)$ and at most three such pairs) of the $E_i$ are resemblers, and we also have shown that at most three of them are envelopers. Thus, there are at most twelve of the sets $D(E_1),\dots,D(E_m)$ in the layer $\Y_{j',k'}$ are removed resemblers or envelopers. So for $m \geq 13$, at least one $E_i$ belongs to $\Z_{j',k'}$.

\end{itemize}
In summary, by choosing $m \geq 13$, for each set $C$ with appropriate $({\mathbf A},{\mathbf B})$-cardinality not containing any $U_{j,k}$, there is a correspondence set $D(E_i)$ in $\Y_{j',k'}$ which was not removed in constructing $\Z_{j',k'}$.
\end{proof}

\begin{proposition}\label{proposition:2ndcovering-converse}
Let $C$ be a subset of $\mathbf{A} \cup \mathbf{B}$ such that $\left| C \right| \leq d-1$ and such that $C$ does not contain any of the sets $U_{j, k} \in \mathcal{U}$. Let $\overleftarrow{\jmath} \leq j' \leq \overrightarrow{\jmath}$ and $1 \leq k' \leq \delta_{j'}$.  
If any $2d$-set $S \in \Z_{j',k'}$ satisfies the condition $\pi(S) \supset C$, then either:
\begin{itemize}
\item $\left| C^A \right| \leq d-j'$ and $\left| C^B \right| \leq j'$, or
\item $\left| C^A \right| \leq d-j'$, $\left| C^B \right| = j' + 1$, and $k' = \delta_j'$. 
\end{itemize}
\end{proposition}
\begin{proof}
 A $2d$-set $S$ in $\Z_{j',k'}$ is either a correspondence set, an interpolation set $I_{j',k'}$ where $k' < \delta_{j'}$, or an interpolation set $I_{j',\delta_{j'}}$.

First suppose $S$ is a correspondence set in $\Z_{j',k'}$ such that $\pi(S) \supseteq C$. 
Then $\left| C^A \right| \leq \Awidth(S) = d-j'$ 
and $\left| C^B \right| \leq \Bwidth(S) = j'$. 

Next suppose that $S = I_{j',k'}$ is an interpolation set and that $k' \neq \delta_{j'}$. 
Then $\left| C^A \right| \leq \Awidth(S) = d-j'+1$ and $\left| C^B \right| \leq \Bwidth(S) = j'$. 
So if $\left| C^A \right| \leq d-j'$, then we're done. 
Otherwise, $\left| C^A \right| = d-j'+1$ and in fact $C^A = \pi(I_{j',k'}) \cap \mathbf{A}$. 
But by Lemma~\ref{lemma:2nd-interp-properties}(\ref{piIjk-is-Ujk}), $\pi(I_{j,k}) \cap \mathbf{A} = U_{j',k'}$, contradicting the assumption that $C$ does not contain any such set. 

Finally, suppose that $k' = \delta_{j'}$ and that $S=I_{j',\delta_{j'}} \supset C$.
Then $\left| C^A \right| \leq \Awidth(S) = d-j'$ and $\left| C^B \right| \leq \Bwidth(S) = j'+1$. 
So $C$ satisfies the desired properties.   
\end{proof}

\begin{corollary} \label{corollary:2ndcovering}
Let $S$ be a subset of $[2] \times (\mathbf{A} \cup \mathbf{B})$ such that $\pi(S)$ does not contain any of the sets $U_{j, k}$. Then $S$ is active on layer $\Z_{j', k'}$ if and only if one of the following conditions hold:
\begin{itemize}
\item $\Bwidth(S) \leq j' \leq d-\Awidth(S)$, or
\item $\Bwidth(S)-1 \leq j' \leq d-\Awidth(S)$ and $k' = \delta_{j'}$.
\end{itemize}
\end{corollary}
\begin{proof}
Let $C = \pi(S)$. If $\Bwidth(S) \leq j' \leq d-\Awidth(S)$, then a correspondence set in $\Z_{j',k'}$ contains $D(C)$ which in turn contains $S$ by Proposition~\ref{proposition:2ndcovering}. Conversely, if there is a $2d$-set $T \in \Z_{j', \ell'}$ such that $T \supset S$, then $\pi(T) \supseteq \pi(S) = C$. So by Proposition~\ref{proposition:2ndcovering-converse}, either $\Bwidth(S) \leq j' \leq d-\Awidth(S)$ or $\Bwidth(S)-1 \leq j' \leq d-\Awidth(S)$ and $k' = \delta_{j'}$.
\end{proof}

Restated, Corollary~\ref{corollary:2ndcovering} says that the layers of $\Zn$ on which $S$ is active comprise a sequence of complete and consecutive sections. 

\begin{proposition} \label{prop:2nddimensionreduction}
The family $\Zn$ has dimension reduction.
\end{proposition}

\begin{proof}
We will show that for every set $S \subseteq [2] \times (\mathbf{A} \cup \mathbf{B})$, the set of layers on which $S$ is active is an interval in $\Zn$.

Recall that $\width(S) = \left| \pi(S) \right|$, $\Awidth(S) = \left| \pi(S) \cap \mathbf{A} \right|$, and $\Bwidth(S) = \left| \pi(S) \cap \mathbf{B} \right|$.  

\begin{itemize}

\item
\textbf{Case 1:} Suppose $\width(S) = d+2$. 

Since all correspondence sets have width $d$ and all interpolation sets have width $d+1$, thus $S$ is not active on any layer of $\Zn$.  

\item
\textbf{Case 2:} Suppose $\width(S) = d+1$. 

Then $S$ is not contained in any correspondence set and by Corollary~\ref{corollary:I-intersection}, $S$ is contained in at most one interpolation set. So $S$ is active on either no layers or exactly one layer of $\Zn$.  

\item
\textbf{Case 3:} Suppose $\width(S) = d$. 

Then by Lemma~\ref{lemma:2ndadjacentlayerZ}, $S$ is active on either no layers, one layer, or two adjacent layers of $\Zn$. 

\item
\textbf{Case 4:} Suppose $\width(S) \leq d-1$ and $(\pi(S) \cap \mathbf{A}) \supseteq U_{j,k}$ for some $j$ and $k$. 

Let $C = \pi(S)$. If $D(E)$ is a correspondence set in $\Yn$ such that $E \supseteq C$, then $D(E)$ is an $I_{j,k}$-enveloper, and thus does not appear anywhere in $\Zn$. So we only need to show that no interpolation set other than $I_{j,k}$ satisfies the property that  its projection contains $C$. 

First, suppose $I_{j'',k''} \supseteq C$ for some $(j'',k'') \neq (j,k)$ with $k'' < \delta_{j''}$. Then
\begin{align*}
P_{j'',k''} \cup \{f_{j''},g_{j''}\} & =  U_{j'',k''} \\ 
  & =  \pi(I_{j'',k''}) \cap \mathbf{A} \\
& \supseteq  C \cap \mathbf{A} \\
& \supseteq  U_{j,k} \\
& \supseteq  P_{j,k} 
\end{align*} 
and so $ \left| P_{j,k} \setminus P_{j'',k''} \right| \leq 2$, 
violating Lemma~\ref{lemma:Pjk-diff}. 

On the other hand, suppose $\pi(I_{j'',\delta_{j''}}) \supseteq C$ for some $j''$. We again calculate 
\begin{align*}
P_{j'',k''} \cup \{f_{j''}\} & =  U_{j'',k''} \\ 
  & =  \pi(I_{j'',k''}) \cap \mathbf{A} \\
& \supseteq  C \cap \mathbf{A} \\
& \supseteq  U_{j,k} \\
& \supseteq  P_{j,k} 
\end{align*} 
and so $ \left| P_{j,k} \setminus P_{j'',k''} \right| \leq 1$,
again violating Lemma~\ref{lemma:Pjk-diff}.

We conclude that either $S$ is active only on layer $\Z_{j,k}$ or $S$ is active on no layers at all. 

\item
\textbf{Case 5:} Suppose that none of the above cases hold. That is, $\width(S) \leq d-1$ and there are no values of $j, k$ for which $(\pi(S) \cap \mathbf{A}) \supseteq U_{j,k}$. 

Let $q = \Awidth(S)$ and $r=\Bwidth(S)$. 
Then by Corollary~\ref{corollary:2ndcovering}, $S$ is active on all layers of sections $r,r+1,\dots,d-q$ and possibly on $\Z_{r-1,\delta_{r-1}}$, but on no others. These layers form an interval.  

\end{itemize}
In all cases this set of layers on which $S \subseteq [2] \times (\mathbf{A} \cup \mathbf{B})$ is active is either empty, a single layer, two consecutive layers within a section, or all layers within some number of consecutive sections. In particular, the set of layers on which $S$ is active is an interval in $\Zn$, thus $\Zn$ satisfies dimension reduction.
\end{proof}

\begin{theorem}\label{theorem:2ndZsummary}
The layer family $\Zn$ has endpoint count, strong adjacency, and dimension reduction. The diameter of $\Zn$ is in $\Omega(n^2/\log n)$.\hfill\qedsymbol
\end{theorem}
This provides an answer to H\"ahnle's question in~\cite{Haehnle:ConstructingSPGs} of how adjacency and endpoint-count interact with dimension reduction, solves Problem~5.2 in~\cite{Kim:PolyhedralGraphAbstractions}, and answers the question of Eisenbrand et al.~in~\cite{Eisenbrand:Diameter-of-Polyhedra} of whether the diameter is affected by considering all previously-identified combinatorial properties for abstractions of polytopes. For further discussion and open problems about abstractions, we refer the reader to~\cite{Eisenbrand:CommentsRecentProgress} by Eisenbrand or the extensive exposition in Section 3 of~\cite{Santos:RecentProgress} by Santos.

\appendix

\section{Random collections of sets}\label{section:separation-lemmas}

In this section we prove certain properties of random collections of sets that we needed in the previous two sections. We perform our own calculation inspired by some results found in the literature, since the exercise by Feller (see~\cite{Feller}) is only applicable when the number $\mathsf{n}$ of samples drawn is constant.

Fix $0 < p_1, p_2 < 1$ and let $q_i = 1-p_i$ for $i=1,2$. Let $A$ and $B$ be subsets of $[N]=\{1,\dots,N\}$ chosen uniformly at random with $\left| A \right| = Np_1$, $\left| B \right| = Np_2$. Here and throughout this appendix, it is necessary to introduce appropriate floors and ceilings in various statements since the values $p_iN$ are not integers in general. The asymptotic results which we wish to establish in Proposition~\ref{proposition:randomintersection} and Theorem~\ref{theorem:randomintersection} are not affected by this technicality, and to improve readability, the appropriate floors and ceilings have not been written in the exposition.

By first fixing $A$ and then choosing $B$, we see that the size of the intersection is given by the \emph{hypergeometric distribution}:
\[\prob \left(\left| A \cap B \right| = k \right) = \prob(\mathbb{H}_{N,p_1,p_2} = k) =H_{N,p_1,p_2}(k) := \frac{\displaystyle\binom{Np_1}{k}\binom{Nq_1}{Np_2-k}}{\displaystyle\binom{N}{Np_2}},\]
where $\mathbb{H}_{N,p_1,p_2}$ is the hypergeometric random variable of the number ($k$) of successes when sampling $p_2N$ times without replacement from a population of $N$ with a success ratio of $p_1$.

\begin{lemma}\label{lemma:hypergeometric-tail}
Let $\mathsf{H}$ be the hypergeometric random variable of the number of successes when sampling $\mathsf{n}$ times from a population where $\mathsf{M}$ of the $\mathsf{N}$ elements of the population are considered successful. If ${\mathsf t} \geq 0$, then
\begin{enumerate}
\item\label{right-tail}
The probability that $\mathsf{H} \geq (\frac{\mathsf{M}}{\mathsf{N}}+{\mathsf t})\mathsf{n}$ is at most $e^{-2\mathsf{t}^2\mathsf{n}}$.
\item\label{left-tail}
The probability that $\mathsf{H} \leq (\frac{\mathsf{M}}{\mathsf{N}}-{\mathsf t})\mathsf{n}$ is at most $e^{-2\mathsf{t}^2\mathsf{n}}$.
\end{enumerate}\hfill\qedsymbol
\end{lemma}
The upper bound in part~(\ref{right-tail}) on the probability of the hypergeometric tail is found in the note~\cite{Chvatal} by Chv\'atal, follows from the discussion in Section~6 of~\cite{Hoeffding} by Hoeffding, and is presented as inequality (9) in~\cite{Skala} by Skala. (This inequality for the right tail of the hypergeometric random variable is the analogue to similar results for the binomial random variable found in Theorem~1 of~\cite{Hoeffding} and Theorem~A.1.4 in~\cite{AlonSpencer}.) Part~(\ref{left-tail}) on the left tail can be immediately proved from part~(\ref{right-tail}), which is found as inequality~(14) in~\cite{Skala} by Skala.

\begin{proposition}\label{proposition:randomintersection}
Fix $0 < p_1, p_2 < 1$ and $\varepsilon > 0$. Suppose that $A_i$ is a set of cardinality $Np_i$ for $i=1,2$. Then 
\[ \prob\left( \left| A_1 \cap A_2 \right| > N(p_1 p_2 + \varepsilon) \right)\]
and
\[ \prob\left( \left| A_1 \cap A_2 \right| < N(p_1 p_2 - \varepsilon) \right)\]
decay exponentially in $N$, with all other parameters ($p_1,p_2,\varepsilon$) fixed. In particular, for any polynomial $f$, 
\[ \lim_{N \rightarrow \infty}f(N) \prob\left( \left| A_1 \cap A_2 \right| > N(p_1 p_2 + \varepsilon) \right) = \lim_{N \rightarrow \infty}f(N) \prob\left( \left| A_1 \cap A_2 \right| < N(p_1 p_2 - \varepsilon) \right) = 0.\]
\end{proposition}

\begin{proof}
%Let $\mathbb{H}_{N,p_1,p_2}$ be the hypergeometric random variable of the number of successes when sampling $p_2N$ times without replacement from a population of $N$ with a success ratio of $p_1$.
%The probability mass function of $\mathbb{H}_{N,p_1,p_2}$ is
%\[ H_{N,p_1,p_2}(k) = \frac{\displaystyle\binom{Np_1}{k}\binom{Nq_1}{Np_2-k}}{\displaystyle\binom{N}{Np_2}}=\prob\left( \left| A_1 \cap A_2 \right|=k \right).\]
By applying Lemma~\ref{lemma:hypergeometric-tail}(\ref{right-tail}) with $\mathsf{n}=Np_2$, $\mathsf{M}=Np_1$, $\mathsf{N}=N$, and ${\mathsf t} = \frac{\varepsilon}{p_2}$ and using $N(p_1 p_2 + \varepsilon) = \mathsf{n}({\textstyle \frac{\mathsf{M}}{\mathsf{N}}} + \mathsf{t})$,
\begin{align*}
\prob\left( \left| A_1 \cap A_2 \right| > N(p_1 p_2 + \varepsilon) \right) 
&= \prob\left( \mathbb{H}_{N,p_1,p_2} > N(p_1 p_2 + \varepsilon) \right)\\
&= \prob\left( \mathsf{H} > \mathsf{n}({\textstyle \frac{\mathsf{M}}{\mathsf{N}}} + \mathsf{t}) \right)\\
&\leq e^{-2\mathsf{t}^2\mathsf{n}}\\
&= e^{-2\varepsilon^2N/p_2}. 
\end{align*}
Similarly, by applying Lemma~\ref{lemma:hypergeometric-tail}(\ref{left-tail}) with $\mathsf{n}=Np_2$, $\mathsf{M}=Np_1$, $\mathsf{N}=N$, and ${\mathsf t} = \frac{\varepsilon}{p_2}$ and using $N(p_1 p_2 - \varepsilon) = \mathsf{n}({\textstyle \frac{\mathsf{M}}{\mathsf{N}}} - \mathsf{t})$, we prove that
\[
\prob\left( \left| A_1 \cap A_2 \right| > N(p_1 p_2 + \varepsilon) \right)  \leq e^{-2\varepsilon^2N/p_2}.
\]

For fixed $\varepsilon$ and $p_2$ the probabilities $\prob\left( \left| A_1 \cap A_2 \right| > N(p_1 p_2 + \varepsilon) \right)$ and $\prob\left( \left| A_1 \cap A_2 \right| < N(p_1 p_2 - \varepsilon) \right)$ decay exponentially in $N$. Thus, for any polynomial $f$,
\[\lim_{N \rightarrow \infty}f(N) \prob\left( \left| A_1 \cap A_2 \right| > N( p_2 + \varepsilon) \right)=0.\]
For fixed $\varepsilon$ and $p_2$, a similar result holds for the left tail.
\end{proof}

We now prove a similar result about the intersection of several subsets, each having cardinality a constant fraction of $N$. 
\begin{theorem}\label{theorem:randomintersection}
Fix $0 < p_1, \dots, p_\ell < 1$ and $\varepsilon > 0$. 
Suppose that $A_i$ is a set of cardinality $Np_i$ for $i=1,\dots,\ell$.
Then 
\[ \prob\left( \left| A_1 \cap \dots \cap A_\ell \right| > N(p_1 \dots p_\ell + \varepsilon) \right)\]
and
\[ \prob\left( \left| A_1 \cap \dots \cap A_\ell \right| < N(p_1 \dots p_\ell - \varepsilon) \right)\]
both decay exponentially in $N$, when the parameters $p_1,\dots,p_\ell$, and $\varepsilon$ are fixed. In particular, for any polynomial $f$,
\[
\lim_{N \rightarrow \infty}f(N) \prob\left( \left| A_1 \cap \dots \cap A_\ell \right| > N(p_1 \dots p_\ell + \varepsilon) \right) =0
\]
and
\[
\lim_{N \rightarrow \infty}f(N) \prob\left( \left| A_1 \cap \dots \cap A_\ell \right| < N(p_1 \dots p_\ell - \varepsilon) \right) =0.
\]
\end{theorem}

\begin{proof}
We prove the result on induction on $\ell \geq 2$. Let $M=A_1 \cap \dots \cap A_\ell$, 
and for the given $\varepsilon$, define $\displaystyle\varepsilon' = \varepsilon(2p_{\ell+1})^{-1}>0$.
We show that
\[ \prob\left( \left| A_1 \cap \dots \cap A_{\ell+1} \right| > N(p_1 \dots p_{\ell+1} + \varepsilon) \right)\]
decays exponentially in $N$. This probability can be decomposed as the probability of disjoint events:
\begin{equation}\label{equation:right-tail-full}
\sum_{k=0}^N \prob(|M|=k) \cdot \prob\left[ |M \cap A_{\ell+1}| > N(p_1 \dots p_{\ell+1} + \varepsilon)  \,\Big|\, |M|=k \right].
\end{equation}
Consider the terms in~\eqref{equation:right-tail-full} where $k > N(p_1 \dots p_\ell + \varepsilon')$, namely
\[
\sum_{k>N(p_1 \dots p_\ell + \varepsilon')} \prob(|M|=k) \cdot \prob\left[ |M \cap A_{\ell+1}| > N(p_1 \dots p_{\ell+1} + \varepsilon)  \,\Big|\, |M|=k \right],
\]
which is bounded above by $\prob(|M| > N(p_1 \dots p_\ell + \varepsilon'))=\prob(|A_1 \cap \dots \cap A_\ell| > N(p_1 \dots p_\ell + \varepsilon'))$, and by induction on $\ell$, this decays exponentially in $N$ as $p_1,\dots,p_\ell$, and $\varepsilon'=\varepsilon(2p_{\ell+1})^{-1}$ remain fixed. Similarly the terms in~\eqref{equation:right-tail-full} where $k < N(p_1 \dots p_\ell - \varepsilon')$, which is
\[
\sum_{k<N(p_1 \dots p_\ell - \varepsilon')} \prob(|M|=k) \cdot \prob\left[ |M \cap A_{\ell+1}| > N(p_1 \dots p_{\ell+1} + \varepsilon)  \,\Big|\, |M|=k \right]
\]
is bounded above by $\prob(|M| < N(p_1 \dots p_\ell - \varepsilon'))$, which decays exponentially in $N$. The remaining terms in~\eqref{equation:right-tail-full} are
\begin{equation}\label{equation:right-tail-middle}
\sum_{k=N(p_1 \dots p_\ell - \varepsilon')}^{N(p_1 \dots p_\ell + \varepsilon')} \prob(|M|=k) \cdot \prob\left[ |M \cap A_{\ell+1}| > N(p_1 \dots p_{\ell+1} + \varepsilon)  \,\Big|\, |M|=k \right].
\end{equation}
The probability
\[
\prob\left[ |M \cap A_{\ell+1}| > N(p_1 \dots p_{\ell+1} + \varepsilon)  \,\Big|\, |M|=k \right]
\]
is an increasing function of $k$, thus by using the maximal applicable $k$-value of $N(p_1 \dots p_{\ell} + \varepsilon')$, the quantity in~\eqref{equation:right-tail-middle} is bounded above by
\begin{equation}\label{equation:right-tail-middle-bound}
2N\epsilon' \cdot \prob\left[ |M \cap A_{\ell+1}| > N(p_1 \dots p_{\ell+1} + \varepsilon)  \,\Big|\, |M|=N(p_1 \dots p_{\ell} + \varepsilon') \right].
\end{equation}
By applying Lemma~\ref{lemma:hypergeometric-tail}(\ref{right-tail}) with $\mathsf{n}=Np_{\ell+1}$, $\mathsf{M}=|M|=N(p_1 \dots p_{\ell} + \varepsilon')$, $\mathsf{N}=N$, and ${\mathsf t} = \frac{\varepsilon-\varepsilon' p_{\ell+1}}{p_{\ell+1}} > 0$ and using $N(p_1 \dots p_{\ell+1} + \varepsilon) = \mathsf{n}({\textstyle \frac{\mathsf{M}}{\mathsf{N}}} + \mathsf{t})$, the expression in~\eqref{equation:right-tail-middle-bound} is bounded above by
\[
2N\epsilon' \cdot \exp\left\{-2 \left(\frac{\varepsilon-\varepsilon' p_{\ell+1}}{p_{\ell+1}}\right)^2 \cdot Np_{\ell+1}\right\},
\]
which decays exponentially in $N$ when all other parameters are fixed.

Now we show that
\[ \prob\left( \left| A_1 \cap \dots \cap A_{\ell+1} \right| < N(p_1 \dots p_{\ell+1} - \varepsilon) \right)\]
decays exponentially in $N$, pointing out the differences to the previous proof. This probability can be decomposed as
\begin{equation}\label{equation:left-tail-full}
\sum_{k=0}^N \prob(|M|=k) \cdot \prob\left[ |M \cap A_{\ell+1}| < N(p_1 \dots p_{\ell+1} - \varepsilon)  \,\Big|\, |M|=k \right].
\end{equation}
As before, the sum of the terms in~\eqref{equation:left-tail-full} where $k > N(p_1 \dots p_\ell + \varepsilon')$ or where $k < N(p_1 \dots p_\ell - \varepsilon')$ both decay exponentially in $N$. The remaining terms in~\eqref{equation:left-tail-full} are
\begin{equation}\label{equation:left-tail-middle}
\sum_{k=N(p_1 \dots p_\ell - \varepsilon')}^{N(p_1 \dots p_\ell + \varepsilon')} \prob(|M|=k) \cdot \prob\left[ |M \cap A_{\ell+1}| < N(p_1 \dots p_{\ell+1} - \varepsilon)  \,\Big|\, |M|=k \right].
\end{equation}
which is bounded above by
\begin{equation}\label{equation:left-tail-middle-bound}
2N\epsilon' \cdot \prob\left[ |M \cap A_{\ell+1}| < N(p_1 \dots p_{\ell+1} - \varepsilon)  \,\Big|\, |M|=N(p_1 \dots p_{\ell} - \varepsilon') \right],
\end{equation}
since the probability that $|M \cap A_{\ell+1}| < N(p_1 \dots p_{\ell+1} - \varepsilon)$ increases as $k=|M|$ decreases.

By applying Lemma~\ref{lemma:hypergeometric-tail}(\ref{left-tail}) with $\mathsf{n}=Np_{\ell+1}$, $\mathsf{M}=|M|=N(p_1 \dots p_{\ell} - \varepsilon')$, $\mathsf{N}=N$, and ${\mathsf t} = \frac{\varepsilon-\varepsilon' p_{\ell+1}}{p_{\ell+1}} > 0$ and using $N(p_1 \dots p_{\ell+1} - \varepsilon) = \mathsf{n}({\textstyle \frac{\mathsf{M}}{\mathsf{N}}} - \mathsf{t})$, the expression in~\eqref{equation:left-tail-middle-bound} is bounded above by
\[
2N\epsilon' \cdot \exp\left\{-2 \left(\frac{\varepsilon-\varepsilon' p_{\ell+1}}{p_{\ell+1}}\right)^2 \cdot Np_{\ell+1}\right\},
\]
which decays exponentially in $N$ as well.
\end{proof}

We now list the specific consequences that we exploited in the previous sections. 

The first result is quite weak and easy to prove directly. Let $P_1, \dots, P_{\delta-1}$ be the sets defined in Section~\ref{section:Eisenbrand1}. Each of these sets is a subset of $[n]$, so we take $N:=n=2d$.   
 
\begin{proposition} \label{proposition:P-intersection}  
For $j \neq j'$, the sets $P_j$ and $P_{j'}$ intersect in at most $d-4$ elements. 
\end{proposition} 

\begin{proof}
The probability that the desired event does not happen is at most the sum of the $\binom{\delta}{2}$ individual probabilities that $\left| P_j \cap P_{j'} \right| \geq d-4$. Each set is of size $d = \frac{1}{2}N$. Also $\delta < N$ and so $\binom{\delta}{2} < N^4$. By applying Proposition~\ref{proposition:randomintersection} with $p_1 = p_2 = \frac{1}{2}$, $\varepsilon = \frac{1}{8}$, and $f(N) = N^4$, we get with probability approaching one that for all choices of $j, j'$, 
\[\left| P_j \cap P_{j'} \right| 
< \left( \left( \frac{1}{2} \right)^2 + \frac{1}{8} \right) N
= \frac{3}{8} N 
< \frac{1}{2} N - 4 
= d - 4.\]
\end{proof}

Of course the same argument will work if $d$ is any positive fraction of $N$ and we replace $d-4$ by $d-k$ for any constant $k$. 

We now consider the sets $P_{j,k}$ and $Q_{j,k}$ defined in Section~\ref{section:Eisenbrand2}. Since $P_{j,k} \subseteq \mathbf{A}$, $Q_{j,k} \subseteq \mathbf{B}$, and $\left| \mathbf{A} \right| = \left| \mathbf{B} \right| = \frac{n}{2}$, we take $N:=\frac{n}{2}=2d$. 

\begin{lemma}\label{lemma:Pjk-union}
With probability tending to one, the union of each four of the sets $P_{j,k}$ contains at least $\frac{13d}{16}$ elements.
With probability tending to one, the union of each four of the sets $Q_{j,k}$ contains at least $\frac{13d}{16}$ elements.
\end{lemma}

\begin{proof}
We prove the result for the $Q_{j,k}$ sets only because the indexing matches the cardinality. By symmetry, the same result holds for the random $(d-j-1)$-subsets $P_{j,k}$ and the random $j$-subsets $Q_{j,k}$. 

The total number of sets $Q_{j,k}$ is the diameter $\delta$ of the entire family, which is less than $N^2$. The probability that the desired event does not happen is at most the sum of the $\binom{\delta}{4} < N^8$ individual probabilities that $\left| Q_{j_1, k_1} \cup \dots \cup Q_{j_4. k_4} \right| \geq \frac{13d}{16} = \frac{13}{32}N$ does not occur. For $i = 1,2,3,4$, let $Q_i'$ be a random $\frac{d}{4}$-subset of $Q_{j_i, k_i}$ and let $A_i = (Q_i')^c \subseteq \mathbf{B}$. Then $A_i$ is a random set of size $N - \frac{d}{4} = \frac{7N}{8}$. The expected size of $A_1 \cap A_2 \cap A_3 \cap A_4$ is $\left( \frac{7}{8} \right)^4 N = \frac{2401}{4096}N$. So by applying Theorem~\ref{theorem:randomintersection} with $m=4$, $p_1=p_2=p_3=p_4 = \frac{7}{8}$, $\varepsilon = \frac{31}{4096}$, and $f(N) = N^4$, we get that with probability approaching one, for all choices of $(j_1, k_1), \dots, (j_4, k_4)$ we have
\begin{align*}
\left| Q_{j_1, k_1} \cup \dots \cup Q_{j_4,k_4} \right| 
& \geq \left| Q_1' \cup Q_2' \cup Q_3' \cup Q_4'  \right| \\
& = \left| A_1^c \cup A_2^c \cup A_3^c \cup A_4^c \right| \\
& = N - \left| A_1 \cap A_2 \cap A_3 \cap A_4 \right| \\
& > N - \left( \frac{2401}{4096} + \frac{31}{4096} \right) N \\
& = \frac{13}{32}N,
\end{align*}
which equals $\frac{13d}{16}$.
\end{proof}

\begin{lemma}\label{lemma:Pjk-diff}
For $(j,k) \neq (j',k')$, we have $\left| P_{j,k} \setminus P_{j',k'} \right| \geq 5$.
For $(j,k) \neq (j',k')$, we have $\left| Q_{j,k} \setminus Q_{j',k'} \right| \geq 5$.
\end{lemma}

\begin{proof}
As before, we prove the result for the $Q_{j,k}$ sets only. The same result holds by symmetry for the sets $P_{j,k}$. The number of pairs $(j,k), (j',k')$ is at most $\delta^2 < N^4$. The minimum size of a given $Q_{j,k}$ is $\frac{d}{4} = \frac{N}{8}$ and the maximum size of $Q_{j',k'}$ is $\frac{3d}{4} = \frac{3}{8}N$. For a given pair, let $A_1$ be a random $\frac{1}{8}N$-subset of $Q_{j,k}$ and
$A_2$ be a random $\frac{3}{8}N$-superset of $Q_{j',k'}$. The expected size of $A_1 \cap A_2$ is $\frac{3}{64}N$. By applying Proposition~\ref{proposition:randomintersection} with $p_1 = \frac{1}{8}$, $p_2 = \frac{3}{8}$, $\varepsilon = \frac{1}{64}$, and $f(N) = N^4$,
we get with probability approaching 1 that for all choices of $(j,k)\not=(j',k')$,
\begin{align*} 
\left| Q_{j,k} \setminus Q_{j',k'} \right| 
& \geq \left| A_1 \setminus A_2 \right| \\ 
& = \left|  A_1 \right|  - \left| A_1 \cap A_2 \right| \\
& \geq \frac{1}{8}N - \left( \frac{3}{64} + \frac{1}{64} \right) N \\
& = \frac{1}{16} N,
\end{align*}
which is more than 5 (or any other constant) for $N$ sufficiently large. 
\end{proof}

The events in Lemmas~\ref{lemma:Pjk-union} and~\ref{lemma:Pjk-diff} hold simultaneously using the union bound, since the probability of each approaches $1$.

\section*{Acknowledgements}

The authors are grateful for the helpful comments from referees.

%%%%%%%%%%%%%%%%%%%%%%%%%%%%%%%%%%%%%%%%%%%%%%%%%%%%%%%%%%%%

\end{document}